\documentclass[12pt,reqno]{amsart}
\usepackage{fullpage}
\usepackage{amsmath}
\usepackage{mathrsfs,amssymb,graphicx,verbatim}
\usepackage{paralist}

\usepackage{ifthen}

\newcommand{\ifsodaelse}[2]{\ifthenelse{\isundefined{\SODAF}}{#2}{#1}}

\ifsodaelse    {\usepackage{ltexpprt}\usepackage{amsmath}}{}
\usepackage{mathrsfs,amssymb,graphicx,verbatim}
\usepackage{paralist}
\usepackage[breaklinks,colorlinks,pdfstartview=FitH]{hyperref}

\newcommand\remove[1]{}

\newcommand{\rnote}[1]{}
\newcommand{\jnote}[1]{}

\def\cprime{$'$}
\newcommand{\1}{\mathbf{1}}
\newcommand{\e}{\varepsilon}
\newcommand{\R}{\mathbb{R}}
\newcommand{\E}{\mathbb{E}}
\newcommand{\FF}{\mathbb{F}}
\newcommand{\N}{\mathbb{N}}

\renewcommand{\P}{\mathscr{P}}

\newcommand{\pad}{\mathrm{pad}}

\newcommand{\C}{\mathbb{C}}

\newcommand\F{{{\mathscr F}}}

\DeclareMathOperator{\diam}{diam}

\newtheorem{theorem}{Theorem}[section]
\newtheorem{lemma}[theorem]{Lemma}

\newtheorem{claim}[theorem]{Claim}

\newtheorem{corollary}[theorem]{Corollary}

\newtheorem{remark}{Remark}[section]

\newtheorem{conjecture}{Conjecture}

\date{}

\newcommand\Z{{{\mathbb Z}}}
\newcommand\B{{{\mathscr B}}}
\newcommand\id{{{\operatorname{id}}}}
\newcommand\eps{\varepsilon}

  \newtheorem{proposition}[subsection]{Proposition}

\include{psfig}

\title[Localization of maximal inequalities]{Random Martingales and localization of\\ maximal inequalities}

\author{Assaf Naor}
\author{Terence Tao}
\date{}

\begin{document}
\maketitle

\begin{abstract}  Let $(X,d,\mu)$ be a metric measure space. For $\emptyset\neq R\subseteq (0,\infty)$ consider the
 Hardy-Littlewood maximal operator
$$ M_R f(x) \stackrel{\mathrm{def}}{=} \sup_{r \in R} \frac{1}{\mu(B(x,r))} \int_{B(x,r)} |f|\ d\mu.$$
We show that if there is an $n>1$ such that one has the ``microdoubling condition'' $
\mu\left(B\left(x,\left(1+\frac{1}{n}\right)r\right)\right)\lesssim \mu\left(B(x,r)\right)
$
for all $x\in X$ and $r>0$, then the weak $(1,1)$ norm of $M_R$ has the following localization property:
$$
\left\|M_R\right\|_{L_1(X) \to L_{1,\infty}(X)}\asymp \sup_{r>0} \left\|M_{R\cap [r,nr]}\right\|_{L_1(X) \to L_{1,\infty}(X)}.
$$
An immediate consequence is that if $(X,d,\mu)$ is Ahlfors-David $n$-regular then the weak $(1,1)$ norm of $M_R$ is
$\lesssim n\log n$, generalizing a result of Stein and Str\"omberg~\cite{stromberg}. We show that this bound is sharp, by
constructing  a metric measure space $(X,d,\mu)$ that is
Ahlfors-David $n$-regular, for which the weak $(1,1)$ norm of
$M_{(0,\infty)}$ is $\gtrsim n\log n$. The localization property of
$M_R$ is proved by assigning to each $f\in L_1(X)$ a distribution
over {\em random} martingales for which the associated (random) Doob
maximal inequality controls the weak $(1,1)$   inequality for $M_R$.

\end{abstract}

\setcounter{tocdepth}{3}
\tableofcontents

\section{Introduction}

A \emph{metric measure space} $(X,d,\mu)$ is a separable metric space $(X,d)$,  equipped with a Radon measure $\mu$. We assume throughout the non-degeneracy property $0 < \mu(B(x,r)) < \infty$ for all $r>0$, where
$ B(x,r) \stackrel{\mathrm{def}}{=} \{ y \in X: d(x,y) \le r \}.$
For any locally integrable $f: X \to \C$, we can then define \emph{the Hardy-Littlewood maximal function}
$$ Mf(x) \stackrel{\mathrm{def}}{=} \sup_{r > 0} \frac{1}{\mu(B(x,r))} \int_{B(x,r)} |f|\ d\mu,$$
which is easily verified to be measurable.

We shall study the \emph{weak $(1,1)$ operator norm}
of $M$, defined as usual to be the least quantity $0 \leq \|M\|_{L_1(X) \to L_{1,\infty}(X)} \leq \infty$ for which one has the distributional inequality
\begin{equation}\label{eq:def weak 11} \| Mf\|_{L_{1,\infty}(X)} \leq \|M\|_{L_1(X) \to L_{1,\infty}(X)} \cdot \|f\|_{L_1(X)}
\end{equation}
for all $f \in L_1(X)$.  Here $L_p(X)$ ($p\ge 1$) denotes the usual Lebesgue space corresponding to the measure $\mu$,
and $L_{p,\infty}(X)$ is the weak $L_p$ norm,
$$ \|f\|_{L_{p,\infty}(X)} \stackrel{\mathrm{def}}{=} \sup_{\lambda > 0} \lambda\cdot  \mu(  |f| > \lambda )^{1/p}.$$
Analogously to~\eqref{eq:def weak 11}, the {\em strong} $(p,p)$ operator norm of $M$ is defined as usual to be the least quantity $0 \leq \|M\|_{L_p(X) \to L_{p}(X)} \leq \infty$ for which
\begin{equation}\label{eq:def strong pp} \| Mf\|_{L_{p}(X)} \leq \|M\|_{L_p(X) \to L_{p}(X)}\cdot \|f\|_{L_p(X)}
\end{equation}
for all $f \in L_p(X)$.

In most cases of interest it is probably impossible to compute $\|M\|_{L_1(X) \to L_{1,\infty}(X)}$ exactly; notable exceptions to this statement are ultrametric spaces, where the weak $(1,1)$ norm of $M$ equals $1$ (we will return to the class of ultrametric spaces presently), and the real line $\R$, equipped with the usual metric and Lebesgue measure, where it was shown by Melas~\cite{Mel03} that the weak $(1,1)$ norm of $M$ equals $\frac{11+\sqrt{61}}{12}$ (the case of the strong $(p,p)$ norm of $M$, $p>1$, when $X=\R$, remains open, but we refer to~\cite{DGS96,GMM99} for some partial results).

In view of these difficulties, it seems more reasonable to ask for estimates on the asymptotic behavior of the various operator norms of maximal functions. Quite remarkably, despite the wide applicability of maximal inequalities, and significant effort by many researchers, even in the simple case when $X$ is the $n$-dimensional Hilbert space $\ell_2^n$ and $\mu$ is Lebesgue measure, it is unknown whether or not the weak $(1,1)$ norm of $M$ is bounded independently of the dimension $n$.

A classical application of the Vitali covering theorem (see for example~\cite{CW71,stein:large,Durrett96,Hei01})
shows that for any $n$-dimensional normed space $X$, the weak $(1,1)$ and strong $(p,p)$ norms of $M$ grow at
most exponentially in $n$. This was greatly improved by Stein and Str\"omberg~\cite{stromberg} to $\|M\|_{L_1(X) \to L_{1,\infty}(X)}=O(n\log n)$ for a general $n$-dimensional normed space, and to the slightly better bound $\|M\|_{L_1(\ell_2^n) \to L_{1,\infty}(\ell_2^n)}=O(n)$ for $n$-dimensional Hilbert space. Until recently, there was no known example of a sequence of $n$-dimensional normed spaces $X_n$ for which $\|M\|_{L_1(X_n) \to L_{1,\infty}(X_n)}$ tends to $\infty$ with $n$. A recent breakthrough of Aldaz~\cite{Ald08} showed that when $X_n=\ell_\infty^n$, i.e., $\R^n$ equipped with the $\ell_\infty$ norm (whose unit ball is an axis parallel cube), $\|M\|_{L_1(X_n) \to L_{1,\infty}(X_n)}$ must tend to $\infty$ with $n$; the best known lower bound~\cite{Aub09} on $\|M\|_{L_1(\ell_\infty^n) \to L_{1,\infty}(\ell_\infty^n)}$ is $(\log n)^{1-o(1)}$. The best known upper estimate for $\|M\|_{L_1(X) \to L_{1,\infty}(X)}$ when $X=\ell_\infty^n$ remains the Stein-Str\"omberg $O(n\log n)$ bound.

As partial evidence that when $X$ is the $n$-dimensional Euclidean space $\ell_2^n$, the weak $(1,1)$ norm
$\|M\|_{L_1(X) \to L_{1,\infty}(X)}$ might be bounded, we can take Stein's theorem~\cite{Stein82} (see also
the appendix of~\cite{stromberg}) which asserts that in the Euclidean case, for $p>1$ we have $\|M\|_{L_p(X) \to L_{p}(X)}\le C(p)$, where $C(p)<\infty$ depends on $p$ but not on $n$. For general $n$-dimensional normed spaces, Stein and Str\"omberg~\cite{stromberg} obtained the bound $\|M\|_{L_p(X) \to L_{p}(X)}\le c(p)n$, while Bourgain~\cite{Bou86,Bou86-2} and Carbery~\cite{Carb86} proved that for any $n$-dimensional normed space, $\|M\|_{L_p(X) \to L_{p}(X)}\le C(p)<\infty$ provided $p>\frac32$. It is unknown whether or not there is some $1<p<\frac32$ for which there exist $n$-dimensional normed spaces $X_n$ such that $\|M\|_{L_p(X_n) \to L_{p}(X_n)}$ is unbounded. This is unknown even for the case of cube averages $X_n=\ell_\infty^n$. It was shown by Bourgain~\cite{Bou87} that $\|M\|_{L_p(X) \to L_{p}(X)}\le C(p,q)$ for all $p>1$ when $X=\ell_q^n$ and $q$ is an even integer, and this was extended by M\"uller to $X=\ell_q^n$ for all $1\le q<\infty$.

A dimension independent bound on $\|M\|_{L_1(\ell_2^n) \to L_{1,\infty}(\ell_2^n)}$ would mean that the classical
Euclidean Hardy-Littlewood maximal inequality is in essence an infinite dimensional phenomenon. This statement is not quite true, since there is no ``Lebesgue measure" on infinite dimensional Hilbert space, but nevertheless, even Stein's dimension independent bound on $\|M\|_{L_p(\ell_2^n) \to L_{p}(\ell_2^n)}$, $p>1$, has interesting infinite dimensional consequences---see for examples Ti\v{s}er's work~\cite{Tis88} on differentiation of integrals with respect to certain Gaussian measures on Hilbert space (provided that the integrand is in $L_p$ for some $p>1$). Moreover, improved bounds on $\|M\|_{L_1(X) \to L_{1,\infty}(X)}$ are clearly of interest since they would yield improved quantitative estimates in the many known applications of the Hardy-Littlewood maximal inequality. As an example, such bounds are relevant for quantitative variants of Rademacher's differentiation theorem for Lipschitz functions, which are used in results on the bi-Lipschitz distortion of discrete nets (see~\cite{Bou87-2,CKN09}).

Bounds on $\|M\|_{L_1(X) \to L_{1,\infty}(X)}$ and $\|M\|_{L_p(X) \to L_{p}(X)}$ have been also
intensively investigated for metric measure spaces other than finite dimensional normed spaces.
Strong $(p,p)$ bounds for free groups (with counting measure) have been established by Nevo and
Stein in~\cite{nevo}. In Section~\ref{nevo-sec} we prove the corresponding weak $(1,1)$ inequality,
 which is nevertheless not sufficient for the purpose of ergodic theoretical applications as in~\cite{nevo};
 see Conjecture~\ref{freeconj} below for more information\footnote{After presenting our work we learned from
 Michael Cowling that the weak $(1,1)$ inequality for the free group can be also deduced from the work of
 Rochberg and Taibleson~\cite{RT91}. Our combinatorial proof in Section~\ref{nevo-sec} is different from the proof
 in~\cite{RT91}, though it is similar to the proof in an unpublished manuscript of Cowling, Meda and Setti,
 which adapts arguments of Str\"omberg~\cite{Str81} in the case of the hyperbolic space.
 We thank Michael Cowling and Lewis Bowen for showing us the Cowling-Meda-Setti manuscript.}.
 In the case of the Heisenberg group $\mathbb H^{2n+1}$, equipped with either the
 Carnot-Carath\'eodory metric or the Koranyi norm (and the underlying measure being the Haar measure),
 dimension independent strong $(p,p)$ bounds have been obtained by Zienkiewicz~\cite{jacek}, and a
 weak $(1,1)$ bound of $O(n)$ was obtained by Li~\cite{Quan09}. It is unclear if these bounds
 generalize to other nilpotent Lie groups (though perhaps similar methods could apply to
 certain two step nilpotent Lie groups, by replacing the use of~\cite{NT97} in~\cite{jacek}
 with the results of~\cite{MS04,Fis06}).

The main result of the present paper implies a general bound for the weak $(1,1)$ norm of the Hardy-Littlewood maximal function on Ahlfors-David $n$-regular spaces; a class of metric measure spaces that contains the examples described above as special cases (except for the case of the free group, which is dealt with separately in Section~\ref{nevo-sec}). Specifically, assume that the metric measure space $(X,d,\mu)$ satisfies the growth bounds
\begin{equation}\label{eq:def AD}
\forall x\in X\ \forall r>0,\quad r^n\le \mu\left(B(x,r)\right)\le Cr^n,
\end{equation}
where $n\ge 2$, and $C$ is independent of $x,r$. Under this assumption, we show that

\begin{equation}\label{eq:introAD}
\|M\|_{L_1(X) \to L_{1,\infty}(X)}=O(n\log n),
 \end{equation}
 where the implied constant depends only on $C$.
 At the same time, we construct for all $n\ge 2$ an Abelian group $G_n$, equipped with a translation invariant metric $d_n$ and a translation invariant measure $\mu_n$, that satisfies~\eqref{eq:def AD} with $C=81$ \footnote{One can modify the argument to make $C$ arbitrarily close to $1$, but we will not do so here as it requires more artificial constructions.}, yet
 \begin{equation}\label{eq:lowerGn}
 \|M\|_{L_1(G_n) \to L_{1,\infty}(G_n)}\gtrsim n\log n.
 \end{equation}
  We can also ensure that for all $p>1$ we have
\begin{equation}\label{eq:Gn}
\|M\|_{L_p(G_n) \to L_{p}(G_n)}\lesssim_p 1.
 \end{equation}
 Here, and in what follows, we use $X \lesssim Y$, $Y \gtrsim X$ to denote the estimate $X \leq CY$ for some absolute constant $C$; if we need $C$ to depend on parameters, we indicate this by subscripts, thus $X \lesssim_p Y$ means that $X \leq C_p Y$ for some $C_p$ depending only on $p$. We shall also use the notation $X\asymp Y$ for $X\lesssim Y\ \wedge\ Y\lesssim X$.

Note that the bound~\eqref{eq:introAD} contains the Stein-Str\"omberg result for $n$-dimensional normed spaces. It also applies to, say, any translation invariant length metric on nilpotent Lie groups\footnote{It seems likely however that the original Stein-Str\"omberg argument can be extended to this setting.}. However, it falls shy (by a logarithmic factor) of the two $O(n)$ results quoted above: for the Euclidean space $\ell_2^n$, and the Heisenberg group $\mathbb H^{2n+1}$. Our lower bound~\eqref{eq:lowerGn} suggests that in order to improve upon the $O(n \log n)$ bound of Stein and Str\"omberg, one must genuinely use the underlying geometry of the normed vector space and not just the metric properties, or the $L_p$ theory.  For instance, to obtain the bound of $O(n)$ in the case of the Euclidean metric in~\cite{stromberg}, it was necessary to exploit the relationship between averaging on balls and the Poisson semigroup, in order that the Hopf-Dunford-Schwartz maximal inequality can be used. A similar strategy was used for the Heisenberg group in~\cite{Quan09}. This type of relationship does not appear to be available for general norms on $\R^n$.

The results presented above are simple corollaries of a general
localization phenomenon for maximal inequalities, which we shall now
describe. In fact, for the bound~\eqref{eq:introAD} to hold true, we
need to assume a condition which is less restrictive than the
Ahlfors-David regularity condition~\eqref{eq:def AD}; in particular
it need not hold for all radii $r$, and it thus also applies  to
discrete groups of polynomial growth, equipped with the word metric
and the counting measure. All of these issues are explained in the
following subsection.


\subsection{Microdoubling  and the localization theorem}\label{sec:micro}

Let $(X,d,\mu)$ be a metric measure space. For $R\subseteq (0,\infty)$ we consider the maximal operator corresponding to radii in $R$, which is defined by
 \begin{equation}\label{eq:def MR}
  M_R f(x) \stackrel{\mathrm{def}}{=} \sup_{r \in R} \frac{1}{\mu(B(x,r))} \int_{B(x,r)} |f|\ d\mu.
  \end{equation}
Thus, using our previous notation, $M=M_{(0,\infty)}$.

We shall say that $(X,d,\mu)$ is {\em $n$-microdoubling with constant $K$} if for all $x\in X$ and all $r>0$ we have
\begin{equation}\label{eq:def micro}
\mu\left(B\left(x,\left(1+\frac{1}{n}\right)r\right)\right)\le KB(x,r).
\end{equation}
The case $n=1$ in~\eqref{eq:def micro} is the classical {\em $K$-doubling} condition
\begin{equation}\label{eq:def doubling}
\forall x\in X\ \forall r>0,\quad\mu\left(B\left(x,2r\right)\right)\le KB(x,r).
\end{equation}
Note that~\eqref{eq:def micro} follows from the Ahlfors-David $n$-regularity condition~\eqref{eq:def AD}, with $K=eC$.
The microdoubling property appeared in various guises in the literature; for example, it follows from a lemma of Colding and Minicozzi~\cite{CM98} (see also Proposition 6.12 in~\cite{Chee99}) that if $(X,d,\mu)$ is a $K$-doubling {\em length} space, then it is also $n$-microdoubling with constant $O(1)$, where $n=e^{K^{O(1)}}$. We note in passing that this exponential dependence on $K$ is necessary, as exhibited by the interval $X=[1,N]$, with the metric inherited from $\R$, and the measure whose density is $\varphi(x)=\frac{1}{x}$; the doubling constant for this length space is of order $\log N$, but it can only be $n$-microdoubling with $n$ a power of $N$.

Our main result is the following localization theorem for maximal inequalities on microdoubling spaces. It deals, for any $1 \leq p < \infty$, with the weak $(p,p)$ norm of $M_R$,
defined as the optimal number $\|M_R\|_{L_p(X)\to L_{p,\infty}(X)}$
for which the distributional inequality
$$ \mu\left( M_R f > \lambda  \right) \leq \frac{\|M_R\|_{L_p(X)\to L_{p,\infty}(X)}^p}{\lambda^p} \| f\|_{L_p(X)}^p$$
holds for all $f \in L_p(X)$ and $\lambda>0$.

\begin{theorem}[Localisation]\label{local}  Fix $n\ge 1$ and $K\ge 5$. Let $(X,d,\mu)$ be a metric measure space satisfying the microdoubling condition~\eqref{eq:def micro}. Fix $\emptyset\neq R \subseteq (0,\infty)$ and $p \geq 1$.  Then we have
 \begin{equation}\label{eq:our localization}\| M_R \|_{L_p(X)\to L_{p,\infty}(X)} \lesssim K + \left(1+\frac{\log\log K}{1+\log n}\right)^{1/p}\sup_{r > 0} \left\| M_{R \cap [r, n r]} \right\|_{L_p(X)\to L_{p,\infty}(X)}.
 \end{equation}
\end{theorem}
\begin{remark} { \em In the converse direction, one trivially has
$$\left\| M_R \right\|_{L_p(X)\to L_{p,\infty}(X)} \geq \sup_{r > 0} \left\| M_{R \cap [r, n r]} \right\|_{L_p(X)\to L_{p,\infty}(X)}.$$
Note that the term $\frac{\log\log K}{1+\log n}$ in~\eqref{eq:our localization} is always at most $\log \log K$. Thus when $K$ is independent of $n$, up to constants, in order to establish a weak $(p,p)$ maximal
inequality for spaces obeying~\eqref{eq:def micro}, it suffices to do so
for scales localized to an interval $[r,nr]$.  In many cases (e.g.
finite-dimensional normed vector spaces) we can also rescale to $r = 1$.  }
\end{remark}

\subsection{Weak $(1,1)$ norm bounds}\label{sec:norm}
To deduce some corollaries of Theorem~\ref{local}, fix an integer $m\in \N$, and note that for all $f\in L_p(X)$ and $r,\lambda>0$ we have,
\begin{multline*}
\mu\left( M_{R\cap [r,nr]} f > \lambda  \right) =\mu\left(\max_{0\le j\le m-1} M_{R\cap\left[rn^{j/m},rn^{(j+1)/m}\right]}f>\lambda\right)\\\le \sum_{j=0}^{m-1} \mu\left( M_{R\cap\left[rn^{j/m},rn^{(j+1)/m}\right]}f>\lambda\right)\le m\max_{0\le j\le m-1} \mu\left( M_{R\cap\left[rn^{j/m},rn^{(j+1)/m}\right]}f>\lambda\right).
\end{multline*}
Thus, under the assumptions of Theorem~\ref{local} (and specializing to $p=1$), we have for every $m\in \N$,
 \begin{equation}\label{eq:sparsification}
 \| M_R \|_{L_1(X)\to L_{1,\infty}(X)} \lesssim K + m\left(1+\frac{\log\log K}{1+\log n}\right)\sup_{r > 0} \left\| M_{R \cap \left[r, n^{1/m} r\right]} \right\|_{L_1(X)\to L_{1,\infty}(X)}.
 \end{equation}
Note that for $m\ge 2n\log n$ we have $n^{1/m}\le 1+\frac{1}{n}$, and hence for all $r>0$,
\begin{equation}\label{eq:bring A}
M_{R \cap \left[r, n^{1/m} r\right]}f\le \frac{1}{\mu(B\left(x,r\right))}\int_{B\left(x,\left(1+\frac{1}{n}\right)r\right)}|f|d\mu
\stackrel{\eqref{eq:def micro}}{\le} K A_{1+\frac{1}{n}}f,
\end{equation}
where $A_r$ is the averaging operator:
\begin{equation}\label{avdef}
 A_r f(x) \stackrel{\mathrm{def}}{=} \frac{1}{\mu(B(x,r))} \int_{B(x,r)} |f|\ d\mu.
 \end{equation}
Under some mild uniformity assumption on $\mu$, the strong $(1,1)$ norm of $A_r$ is bounded for all $r>0$. For example, if $\mu(B(x,r))$ does not depend on $x$ (as is the case for invariant metrics and measures on groups), then a simple application of Fubini's theorem shows that $\|A_r\|_{L_1(X)\to L_1(X)}\le 1$. In fact, if we knew that $\mu(B(x,r))\le K\mu(B(y,r))$ for all $x\in X$ and $y\in B(x,r)$ (which is a trivial consequence of the Ahlfors-David regularity condition~\eqref{eq:def AD}), then we would have by the same reasoning $\|A_r\|_{L_1(X)\to L_1(X)}\le K$. An elegant way to combine this uniformity condition with the microdoubling condition~\eqref{eq:def micro}, is to impose the following condition, which we call {\em strong $n$-microdoubling with constant $K$}:
\begin{equation}\label{eq:def strong micro}
\forall x\in X\ \forall r>0\ \forall y\in B(x,r),\quad \mu\left(B\left(y,\left(1+\frac{1}{n}\right)r\right)\right)\le KB(x,r).
\end{equation}

Thus, by a combination of~\eqref{eq:sparsification} and~\eqref{eq:bring A}, we see that if $(X,d,\mu)$ satisfies~\eqref{eq:def strong micro}, then $\|M\|_{L_1(X)\to L_{1,\infty}(X)}\lesssim_K n\log n$. Similarly, if $R\cap \left[r,n^{1/m}r\right]$ contains at most one point for all $r>0$, then $\|M_R\|_{L_1(X)\to L_{1,\infty}(X)}\lesssim_K m$. This happens in particular if $$R=2^{\Z}=\left\{2^k:\ k\in \Z\right\},$$ and $m\asymp \log n$, proving the following corollary:
\begin{corollary}\label{coro:our norm bounds}
Fix $n\ge 1$ and $K\ge 5$. Let $(X,d,\mu)$ be a metric measure space satisfying the strong $n$-microdoubling condition~\eqref{eq:def strong micro}. Then
\begin{equation}\label{eq:our SS}
\|M\|_{L_1(X)\to L_{1,\infty}(X)}\lesssim_K n\log n,
\end{equation}
\begin{equation}\label{eq:our lacunary}
\left\|M_{2^\Z}\right\|_{L_1(X)\to L_{1,\infty}(X)}\lesssim_K \log n.
\end{equation}
\end{corollary}
The lacunary maximal function $M_{2^\Z}$ was previously studied for
$n$-dimensional normed spaces by Bourgain in~\cite{Bou86-2}, where
he proved that its strong $(p,p)$ norm is bounded by a dimension
independent constant $C_p<\infty$ (recall that for the non-lacunary
maximal function this is only known for $p>\frac32$). The
logarithmic upper bound~\eqref{eq:our lacunary} on the weak $(1,1)$
norm of the lacunary maximal function when $X$ is an $n$-dimensional
normed space was proved by Men{\'a}rguez and Soria in~\cite{MS92}.

In section~\ref{lin-sec} we present a different approach to the proof of Corollary~\ref{coro:our norm bounds},
following an argument of E. Lindenstrauss~\cite{linden}. While it
gives slightly weaker results, and does not yield the localization theorem, this approach is of independent
interest. Moreover, Lindenstrauss' approach is based on a beautiful randomization of the Vitali covering argument, and as such complements our approach to Theorem~\ref{local}, which is based on a random partitioning method that originated in theoretical computer science and combinatorics (an overview of our technique is contained in Section~\ref{sec:methods}).
The maximal functions considered in~\cite{linden} arose
when taking averages over F{\o}lner sequences of an amenable group
action on a measure space, and were thus not directly connected to
the metric questions that are studied in the present paper.
Nevertheless we consider the arguments in Section~\ref{lin-sec} to be essentially the
same as those in~\cite{linden}. We thank Raanan Schul for pointing
out how the  maximal inequality of E. Lindenstrauss implies the
 Hardy-Littlewood maximal inequality under strong microdoubling.

\subsection{Ultrametric approximations: deterministic and random}\label{sec:methods}

Doob's classical maximal inequality for martingales (see Section~\ref{sec:doob}) is perhaps the simplest and most versatile maximal inequality for which the weak $(1,1)$ norm is known exactly (and is equal to $1$). Our proof of Theorem~\ref{local} relates the weak $(1,1)$ inequality for $M$ to the maximal inequality for martingales, by allowing the martingale itself to be a random object. We show that while the weak $(1,1)$ inequality is not itself a martingale inequality, it is possible to associate to each $f\in L_1(X)$ a distribution over {\em random martingales}. These random martingales stochastically approximate $Mf$, in the sense that we can write down a variant of Doob's inequality for each of them, which, under the microdoubling assumption, in expectation yields theorem~\ref{local}. The details are presented in Section~\ref{frag-sec}.

An alternative interpretation of Doob's maximal inequality is that if $(X,d,\mu)$ is a metric measure space, and if  in addition $d$ is an ultrametric, i.e., $d(x,y)\le \max\{d(x,z),d(z,y)\}$  for all $x,y,z\in X$, then $\|M\|_{L_1(X)\to L_{1,\infty}(X)}\le 1$. Indeed, restrict for simplicity to the case of a finite ultrametric, in which case we obtain an induced hierarchical family of partitions of $X$ into balls, where each ball at a given ``level" is the union of balls of smaller radii at the next ``level". This picture immediately shows that by considering the averages of $f$ on smaller and smaller balls, in the ultrametric case we can reduce the weak $(1,1)$ inequality for $Mf$ to Doob's maximal inequality.

Of course, not every metric is an ultrametric, or even close to an ultrametric. Nevertheless, over the previous two decades, researchers in combinatorics and computer science developed methods to associate to a general metric space $(X,d)$ a distribution over {\em random ultrametrics} $\rho$ on $X$, which dominate $d$ and sufficiently approximate it in various senses (depending on the application at hand). Such methods are often also called ``random partitioning methods", in reference to the hierarchical (tree) structure of ultrametrics. This approach originated in the pioneering works of Linial and Saks~\cite{LS93} and Alon, Karp, Peleg and West~\cite{AKPW95}, and has been substantially developed and refined by Bartal~\cite{Bar96,Bar99}. Important contributions of Calinescu, Karloff and Rabani~\cite{CKR04} and Fakcharoenphol, Rao and Talwar~\cite{FRT04} resulted in a sharp form of ``Bartal's random tree method", and our work builds on these ideas. In~\cite{MN06,MN07} such random ultrametrics were used in order to prove maximal-type inequalities of a very different nature (motivated by embedding problems, as ultrametrics are isometric to subsets of Hilbert space~\cite{Lem01}); these results also served as some inspiration for our work.

One should mention here that the idea of relating metrics to ultrametric models is, of course, standard. Hierarchical partitioning schemes are ubiquitous in analysis and geometry (see the discussion of Calder\'on-Zygmund decompositions in~\cite{Stein82}, or, say, Christ's cube construction in~\cite{Chr90}). Proving maximal inequalities by considering certain Hierarchical partitions is extremely natural; a striking example of this type is Talagrand's majorizing measure theorem~\cite{Tal87}, which deals with sharp maximal inequalities for Gaussian processes via a construction of special ultrametrics (the ultrametric approach is explicit in~\cite{Tal87}, and has an alternative later description~\cite{Tal05} via the so called ``generic chaining"; see also~\cite{GZ03}).  Explicit uses of random coverings and partitions in the context of purely analytic problems occurred in E. Lindenstrauss' aforementioned randomization of the Vitali covering argument for the purpose of pointwise theorems for amenable groups~\cite{linden}, and in the work of Nazarov, Treil and Volberg~\cite{NTV03} on $T(b)$ theorems on non-homogeneous spaces. See also~\cite{LN05} for applications to extensions of Lipschitz functions.

\subsection{Lower bounds}

A standard application of the Vitali covering argument (see e.g. \cite{stein:large} or \cite{terasawa}) yields the inequality
\begin{equation}\label{standard}
 \left\| \tilde M f \right\|_{L_{1,\infty}(X)} \leq \|f\|_{L_1(X)},
\end{equation}
where $\tilde M f$ is the modified Hardy-Littlewood maximal operator
$$ Mf(x) \stackrel{\mathrm{def}}{=} \sup_{r > 0} \frac{1}{\mu(B(x,r,r))} \int_{B(x,r)} |f|\ d\mu,$$
and $B(x,r) \subseteq B(x,r,r) \subseteq B(x,2r)$ is the enlarged ball
$$ B(x,r,r) \stackrel{\mathrm{def}}{=} \bigcup_{y \in B(x,r)} B(y,r) = \{ z \in X: d(x,y), d(y,z) \le r \hbox{ for some } y \in X \}.$$
In particular, if we have the doubling condition~\eqref{eq:def doubling}, then
\begin{equation}\label{mok}
 \|M\|_{L_1(X) \to L_{1,\infty}(X)} \leq K.
\end{equation}
The factor $2$ in~\eqref{eq:def doubling} cannot be replaced by any smaller number while still retaining linear behavior in terms of $K$ of the weak $(1,1)$ operator norm; see \cite{sawano}.

In the absence of any further assumptions on the metric measure space, the bound \eqref{mok} is close to sharp:

\begin{proposition}[The star counterexample]\label{patho} Fix $K \geq 1$.  Then there exists a metric measure space obeying \eqref{eq:def doubling} with
$$ \|M\|_{L_1(X) \to L_{1,\infty}(X)} \geq \lfloor K \rfloor - 1.$$
\end{proposition}

\begin{proof} Without loss of generality we may take $K$ to be an integer.  Let $X$ be the ``star'' graph formed by connecting one ``hub'' vertex $v_0$ to $(K-1)^2$ other ``spoke'' vertices $v_1,\ldots,v_{(K-1)^2}$, with the usual graph metric (thus $d(v_0,v_i) = 1$ and $d(v_i,v_j) = 2$ for all distinct $i,j \in \{1,\ldots,(K-1)^2\}$).  Let $\mu$ be the measure which assigns the mass $K-1$ to $v_0$ and mass $1$ to all other vertices; one easily verifies that \eqref{eq:def doubling} holds.  Let $f \in L_1(X)$ be the function which equals $1$ on $v_0$ and vanishes elsewhere.  Then one easily verifies that $\|f\|_{L_1(X)} = K-1$, that $\mu(X) = K(K-1)$, and that $Mf(x) \geq \frac{K-1}{K}$ for all $x \in X$, and the claim follows.
\end{proof}

\begin{remark}\label{euclidean}  {\em One can achieve a similar effect in a high-dimensional Euclidean space $\R^n$.  If we let $X = \{0,e_1,\ldots,e_n\}$ be the origin and standard basis with the usual Euclidean metric and counting measure, then \eqref{eq:def doubling} holds with $K \stackrel{\mathrm{def}}{=} n+1$, while if we let $f$ be the indicator function of $0$, then $Mf(x) \geq \frac{1}{2}$ for all $x \in X$, and so $\|M\|_{L_1(X) \to L_{1,\infty}(X)} \geq \frac{n+1}{2} = \frac{K}{2}$.
A more sophisticated version of this example was observed in
\cite{shd}: if we take $X$ to be the origin $0$, together with a
maximal $1.01$-separated (say) subset of the sphere $S^{d-1}$, then
\eqref{eq:def doubling} holds for $K = |X| \geq C^n$ for some absolute constant
$C>1$, but $\|M\|_{L_1(X) \to L_{1,\infty}(X)} \geq \frac{K}{2}$ by
the same argument as before.  In particular this shows that the
Hardy-Littlewood weak $(1,1)$ operator norm (as well as the $L_p$
operator norm for any fixed $1 < p < \infty$) for measures in $\R^n$
can grow exponentially in the dimension $n$.  In the converse
direction, a well-known application of the Besicovitch covering
lemma \cite{Bes45,Bes46} shows that $\|M\|_{L_1(X) \to
L_{1,\infty}(X)} \leq C^n$ for some absolute constant $C$ whenever
$X$ is a subset of $\R^n$ with the Euclidean metric, and $\mu$ is an
arbitrary Radon measure.  In particular, as observed in \cite{shd},
this shows that the constants in the Besicovitch covering lemma must
grow exponentially in the dimension (see also~\cite{FL94}).}
\end{remark}

\subsubsection{Adding more hypotheses}


Despite the example in Proposition \ref{patho}, we know due to Corollary~\ref{coro:our norm bounds} that in many cases the bound \eqref{mok} can be significantly improved. In particular, a more meaningful variant of Proposition~\ref{patho} would be if we also impose the natural uniformity condition that  $\mu(B(x,r))$ is independent of $x\in X$.
As discussed in Section~\ref{sec:norm}, this immediately implies that the averaging operators $A_r$ given in~\eqref{avdef} are now contractions on $L_1(X)$.
  Thus in order for the weak $(1,1)$ operator norm to be large, one needs to have contributions to the set $\{ Mf > \lambda \}$ from several scales $r$, rather than just a single scale as in Proposition \ref{patho}.


Another hypothesis that one can add, in order to make a potential counter-example more meaningful, is that the maximal operator $M$ is already of strong-type $(p,p)$ for all $1 < p \leq \infty$, as we know to be the case for $X=\ell_2^n$, due to Stein's theorem~\cite{Stein82}.
Finally, we can make the task of bounding the maximal operator easier by replacing $M$ with the lacunary maximal operator $ M_{2^\Z}$.

Our first main construction shows that even with all of these additional hypotheses and simplifications, we still cannot improve significantly upon \eqref{mok}.

\begin{theorem}[Doubling example]\label{dyadic-1}  Let $K \geq 1$.  Then there exists a metric measure space $(X,d,\mu)$ with $X$ an Abelian group and $d,\mu$ translation-invariant, such that the doubling condition \eqref{eq:def doubling} holds, and $\|M\|_{L_p(X)\to L_p(X)}\lesssim_p 1$ holds for all $1 < p \leq \infty$ (with the implied constant independent of $K$), but such that
\begin{equation}\label{mosh}
 \| M_{2^\Z} \|_{L_1(X) \to L_{1,\infty}(X)} \geq \frac{K}{48}.
\end{equation}
\end{theorem}

We prove this theorem in Section \ref{countersec}.  The basic idea is to first build a maximal operator not arising from a metric measure space which is of strong type $(p,p)$ but not of weak type $(1,1)$, and then take an appropriate ``tensor product'' of this operator with a martingale type operator to obtain a new operator which is essentially a lacunary maximal operator associated to a metric measure space.
The constant $48$ in~\eqref{mosh} can of course be improved, but we will not seek to optimize it here.

As stated earlier, we also construct an example of a metric measure space that shows that Corollary~\ref{coro:our norm bounds} is sharp even under the stronger Ahlfors-David regularity condition~\eqref{eq:def AD}.

\begin{theorem}[Ahlfors-David regular example]\label{micro-ex}  Assume that $n \geq 2$.  Then there exists an Abelian group $G$, with invariant measure $\mu$ and an invariant metric $d$, obeying the Ahlfors-David $n$-regularity condition \eqref{eq:def AD} with $K=81$,  such that
\begin{equation}\label{micro-1}
\| M\|_{L_1(G) \to L_{1,\infty}(G)} \gtrsim n \log n,
\end{equation}
and
\begin{equation}\label{micro-2}
 \left\| M_{2^\Z} \right\|_{L_1(G) \to L_{1,\infty}(G)} \gtrsim \log n.
 \end{equation}
Furthermore we have
\begin{equation}\label{micro-3}
\| M \|_{L_p(G) \to L_p(G)} \lesssim_p 1
\end{equation}
for all $1 < p \leq \infty$.
\end{theorem}



\subsection{The example of the infinite tree}\label{sec:tree intro}



The above examples seem to indicate that the weak $(1,1)$ behavior of the Hardy-Littlewood maximal function can
deteriorate substantially when the doubling constant is large, even when assuming good $L_p$ bounds, as well as
uniformity assumptions on the measure of balls.  Nevertheless, there are some interesting examples of metric measure spaces with very poor (or non-existent) doubling properties, for which one still has a weak $(1,1)$ bound.  We give just one example of this phenomenon, namely the infinite regular tree.

\begin{theorem}[Hardy-Littlewood inequality for the infinite tree]\label{free}  Fix an integer $k \geq 2$, and let $T$ be the infinite rooted $k$-ary tree, with the usual graph metric $d$ and counting measure $\mu$.  Then we have
$$ \| M \|_{L_1(T) \to L_{1,\infty}(T)} \lesssim 1$$
(Thus the implied constant is independent of the degree $k$.)
\end{theorem}

We prove this theorem in Section \ref{nevo-sec}.  We remark that the $L_p$ boundedness of this maximal function for $p > 1$ was essentially established by Nevo and Stein in \cite{nevo}.  The argument here proceeds very differently from the usual covering type arguments, which are totally unavailable here due to the utter lack of doubling for this tree.  Instead, we use a more combinatorial argument taking advantage of the ``expander'' or ``non-amenability'' properties of this tree, which roughly asserts that any given finite subset of the tree must have large boundaries at every distance scale.

When $k$ is odd, $T$ is almost\footnote{More precisely, one needs to enlarge the tree at the root to have $k+1$ descendants instead of $k$.  But one can easily check that this change only affects the weak $(1,1)$ norm of the maximal function by a constant at worst.} identifiable with the free group on $\frac{k+1}{2}$ generators.  The above theorem then suggests that a maximal ergodic theorem in $L_1$ should be available for ergodic actions of free groups on measure-preserving systems (the analogous $L_p$ maximal theorems for $p > 1$ being established in \cite{nevo}).  However, the non-amenability of the free group prevents one from applying standard arguments to transfer Theorem \ref{free} to this setting (indeed, our proof of Theorem \ref{free} will rely heavily on this non-amenability).  Thus the following conjecture remains open:

\begin{conjecture}\label{freeconj}  Let $F$ be a finitely generated free group, and let $w \mapsto T_w$ be an ergodic action of $F$ on a probability space $(X, \B, \mu)$.  Then
$$ \left\| \sup_{n \geq 1} \frac{1}{|B(\id,n)|} \sum_{w \in B(\id,n)} |T_w f| \right\|_{L_{1,\infty}(X)} \lesssim \|f\|_{L_1(X)}$$
for all $f \in L_1(X)$, where $B(\id,n)$ is the collection of words in $F$ of length less than $n$.
\end{conjecture}

We remark that by applying the pointwise convergence theorems in \cite{nevo} and a standard density argument, Conjecture \ref{freeconj} would imply the pointwise convergence result
$$ \lim_{n \to \infty} \frac{1}{|B(\id,n)|} \sum_{w \in B(\id,n)} T_w f(x) = \int_X f\ d\mu$$
for all $f \in L_1(X)$ and almost every $x \in X$.  This result is currently known for $f \in L_p(X)$ for $p>1$, due to \cite{nevo}.

\medskip

\noindent {\bf Acknowledgements. }We thank Raanan Schul for pointing out that the Lindenstrauss maximal inequality implies the  Hardy-Littlewood maximal inequality under strong microdoubling, and Zubin Guatam for explaining the proof of the Lindenstrauss maximal inequality. A. N. was supported in
 part by NSF grants  CCF-0635078
and CCF-0832795, BSF grant 2006009, and the Packard Foundation.
T. T. was supported by a grant from the MacArthur foundation, by NSF grant
DMS-0649473, and by the NSF Waterman award.

\section{Doob-type maximal inequalities}\label{sec:doob}


Let $(X,d,\mu)$ be a metric measure space with $\mu(X)<\infty$ (more
generally, the arguments below extend to the $\sigma$-finite case).
If $\F$ is a $\sigma$-algebra of measurable sets in $X$, we let
$L_p(\F)$ denote the space of $L_p(X)$ functions which are
$\F$-measurable.  The orthogonal projection from $L_2(X)$ to the
closed subspace $L_2(\F)$ will be denoted $f \mapsto \E(f|\F)$, and
as is well known it extends to a contraction on $L_p(X)$ for all $1
\leq p \leq \infty$. The following important inequality of Doob is
classical (see~\cite{Doob90,Durrett96}).



\begin{proposition}[Doob's maximal inequality]\label{doob}  Let $\F_0\subseteq \F_1\subseteq \F_2\subseteq \cdots$ be
an increasing sequence of $\sigma$-algebras. Then we have
$$ f \in L_1(X)\implies \left\| \sup_{k\ge 0} \big| \E( f|\F_k) \big| \right\|_{L_{1,\infty}(X)} \leq \|f\|_{L_1(X)},$$
and for $1 < p \leq \infty$,
$$f \in L_p(X)\implies \left\| \sup_{k\ge 0} \big| \E( f|\F_k) \big| \right\|_{L_p(X)} \le \frac{p}{p-1} \|f\|_{L_p(X)}.$$
\end{proposition}

We now establish a variant of this inequality, in which the
expectations $\E(f|\F_k)$ are replaced by more general sublinear
operators.

\begin{theorem}[Modified Doob's inequality]\label{doob2}  Let $\F_0\subseteq \F_1\subseteq \F_2\subseteq \cdots$ be
an increasing sequence of $\sigma$-algebras and fix $1 \leq p <
\infty$.  For each $k \in \N$ let $M_k$ be a sublinear
operator\footnote{By this we mean that $|M_k(f+g)| \leq |M_k(f)| +
|M_k(g)|$ and $|M_k(cf)| = |c| \cdot|M_k f|$ for all functions $f,g$ in
the domain of $M_k$ and all constants $c\in \R$.} defined on $L_p(X)
+ L_\infty(X)$ such that we have the bounds
\begin{equation}\label{mana}
f \in L_p(X)\implies \| M_k f \|_{L_{p,\infty}(X)} \leq A \|
f\|_{L_p(X)},
\end{equation}
and
\begin{equation}\label{mnf}
f \in L_\infty(X)\implies \| M_k f \|_{L_\infty(X)} \leq B \left\|
\E\left(|f|\bigl|\F_k\right)\right\|_{L_\infty(X)}.
\end{equation}
  Suppose also that we have the
localization property
\begin{equation}\label{emf}
f \in L_p(X) + L_\infty(X)\  \wedge\  E_{k} \in \F_{k}\implies
\1_{E_{k}} M_{k+1} f = M_{k+1} \left(\1_{E_{k}} f\right).
\end{equation}
Then we have
$$ \left\| \sup_{k\ge 0} |M_k f| \right\|_{L_{p,\infty}(X)} \leq \left( (2A)^p + (2B)^p \right)^{1/p} \| f \|_{L_p(X)}$$
for all $f \in L_p(X)$.
\end{theorem}

\begin{remark} {\em Observe that the properties \eqref{mnf},
\eqref{emf} (with $B=1$) are satisfied by the projection operator
$M_{k+1} f \stackrel{\mathrm{def}}{=} \E(f|\F)$ whenever $\F_k \subseteq \F\subseteq
\F_{k+1}$. Thus \eqref{mnf}, \eqref{emf} can be viewed together as a
kind of assertion that $M_{k+1}$ lies ``between'' $\F_k$ and
$\F_{k+1}$ in some sense.}
\end{remark}

\begin{proof}  By monotone convergence we may restrict the supremum over $k\ge 0$ to a finite range, say
$0\leq k \leq K$ for some finite $K\in \N$.  We can then assume
without loss of generality that $\F_k$ is the trivial algebra
$\{\emptyset,X\}$ for all $k <0$. By homogeneity it suffices to show
that
\begin{eqnarray}\label{eq:lmabda=1}
f \in L_p(X)\implies \mu\left( \sup_{0\leq k \leq K} |M_k f| > 1
\right) \leq \left( (2A)^p + (2B)^p \right) \int_X |f|^p\ d\mu.
\end{eqnarray}

Fix $f\in L_p(X)$ and note that Doob's maximal inequality implies
that
\begin{eqnarray*}
\mu\left( \sup_{0\le k\le K} \E\left(|f|\bigr|\F_k\right) \geq
\frac{1}{2B}\right) \leq \mu\left(  \sup_{0\le k\le K} \E\left(
|f|^p \bigr|\F_k\right) \geq \frac{1}{(2B)^p} \right)  \leq (2B)^p
\int_X |f|^p\ d\mu.
\end{eqnarray*}
Thus in order to prove~\eqref{eq:lmabda=1} it will suffice to show
that \begin{eqnarray}\label{eq:A} \mu\left( \left\{\sup_{0\leq k
\leq K} |M_k f|
> 1 \right\} \setminus \left\{\sup_{0\le k\le K}
\E\left(|f|\bigr|\F_k\right) \geq \frac{1}{2B} \right\} \right) \leq
(2A)^p \int_X |f|^p\ d\mu.
\end{eqnarray}

Consider the inclusion
\begin{multline}\label{eq:inclusion1}
\left\{\sup_{0\leq k \leq K} |M_k f| > 1 \right\} \setminus
\left\{\sup_{0\le k\le K} \E\left(|f|\bigr|\F_k\right) \geq
\frac{1}{2B} \right\}\\\subseteq \bigcup_{k=0}^K \left\{|M_kf|>1 \
\wedge\  \sup_{0\le j<k} \E\left(|f|\bigr|\F_j\right) <
\frac{1}{2B}\right\}.
\end{multline}
Therefore, if we introduce the sets
$$ A_k \stackrel{\mathrm{def}}{=} X \setminus \bigcup_{0\le j < k} \left
\{ \E\left(|f|\bigr|\F_{j}\right) \geq \frac{1}{2B} \right\}, $$ and
$$ \Omega_k \stackrel{\mathrm{def}}{=} \left\{
\E\left(|f|\bigr|\F_k\right) \geq \frac{1}{2B} \right\} \cap A_k.$$
Then $A_k\in \F_{k-1}$, the sets $\Omega_k$ are disjoint, and
using~\eqref{emf} we see that~\eqref{eq:inclusion1} implies the
inclusion
\begin{multline}\label{eq:inclusion2}
\left\{\sup_{0\leq k \leq K} |M_k f| > 1 \right\} \setminus
\left\{\sup_{0\le k\le K} \E\left(|f|\bigr|\F_k\right) \geq
\frac{1}{2B} \right\}\subseteq \bigcup_{k=0}^K
\left\{|\1_{A_k}M_kf|>1\right\}\\= \bigcup_{k=0}^K
\left\{|M_k(\1_{A_k}f)|>1\right\}.
\end{multline}
On the other hand, from~\eqref{mnf} we have
\begin{multline*}
\left\| M_k( f \1_{A_k \setminus \Omega_k} ) \right\|_{L_\infty(X)}
\leq B \left\| \E\left( |f| \1_{A_k \setminus \Omega_k} \bigr| \F_k
\right) \right\|_{L_\infty(X)} = B \left\|
\E\left(|f|\bigr|\F_k\right) \1_{A_k \setminus \Omega_k}
\right\|_{L_\infty(X)} \\\le B \cdot\frac{1}{2B} = \frac{1}{2}.
\end{multline*}
Hence by the sublinearity of $M_k$ we have the following inclusion
(up to sets of measure zero):
\begin{eqnarray}\label{eq:inclusion3}\left\{ |M_k (f \1_{A_k})|
> 1 \right\} \subseteq \left\{ |M_k (f \1_{\Omega_k})| > \frac{1}{2}
\right\}.\end{eqnarray} Combining~\eqref{eq:inclusion2}
with~\eqref{eq:inclusion3} and the assumption~\eqref{mana}, we
obtain
\begin{multline*}
\mu\left( \left\{\sup_{0\leq k \leq K} |M_k f| > 1 \right\}
\setminus \left\{\sup_{0\le k\le K} \E\left(|f|\bigr|\F_k\right)
\geq \frac{1}{2B} \right\} \right) \le\sum_{k=0}^K \mu\left(  |M_k
(f \1_{\Omega_k})| > \frac{1}{2}  \right)\\\le \sum_{k=0}^K (2A)^p
\int_{\Omega_k} |f|^p\ d\mu
=(2A)^p\int_{\bigcup_{k=0}^K \Omega_k} |f|^pd\mu\le (2A)^p \int_X
|f|^p\ d\mu.
\end{multline*}
This is precisely the estimate~\eqref{eq:A}, as desired.
\end{proof}

\section{Localization of maximal inequalities}\label{frag-sec}

Let $(X,d,\mu)$ be a bounded metric measure space. Given a partition
$\P$ of $X$ and $x\in X$, we denote by $\P(x)$ the unique element of
$\P$ containing $X$. We shall say that a sequence $\{\P_k\}_{k=0}^\infty$ of
partitions of $X$  is a {\em partition tree} if the following conditions hold true:
\begin{itemize}
\item $\P_0$ is the trivial partition $\{X\}$.
\item For every $x\in X$ and $k\in \{0\}\cup\N$ we have
\begin{equation}\label{eq;power 2 decay}
\diam(\P_k(x))\le
\frac{\diam(X)}{2^k}.
\end{equation}
\item For every $k\in \{0\}\cup\N$ the partition $\P_{k+1}$ is
a refinement of the partition $\P_k$, i.e., for every $x\in X$ we
have $\P_{k+1}(x)\subseteq \P_k(x)$.
\end{itemize}
For $\beta>0$, a probability distribution $\Pr$ over partition trees $\{\P_k\}_{k=0}^\infty$ is said to be
\emph{$\beta$-padded} if for every $x\in X$ and every
$k\in \N$,
\begin{eqnarray}\label{eq:padd} \Pr\left[
B\left(x,\frac{\beta\diam(X)}{2^k}\right)\subseteq \P_k(x)\right]\ge \frac12.
\end{eqnarray}
Note that~\eqref{eq:padd} has the following simple consequence,
which we will use later: for every measurable set $\Omega\subseteq
X$ denote
\begin{eqnarray}\label{eq:def pad}
\Omega^{\pad(k)}_\beta\stackrel{\mathrm{def}}{=}\left\{x\in \Omega:\ B\left(x,\frac{\beta\diam(X)}{2^k}\right)\subseteq \P_k(x)\right\}.
\end{eqnarray}
Thus $\Omega^{\pad(k)}_\beta$ is a random subset of $\Omega$. By Fubini's theorem we have:
\begin{eqnarray}\label{eq:expected}
\E
\left[\mu\left(\Omega^{\pad(k)}_\beta\right)\right]=\int_{\Omega}\Pr\left[
B\left(x,\frac{\beta\diam(X)}{2^k}\right)\subseteq \P_k(x)\right]
d\mu(x) \stackrel{\eqref{eq:padd}}{\ge} \frac{\mu(\Omega)}{2}.
\end{eqnarray}

\begin{remark}\label{rem:measurability} {\em In the definitions above we implicitly made the assumptions that certain events are measurable in the appropriate measure spaces. Namely, for~\eqref{eq:padd} we need the event $\left\{B\left(x,\frac{\beta\diam(X)}{2^k}\right)\subseteq \P_k(x)\right\}$ to be $\Pr$-measurable for every $x\in X$ and $k\in \{0\}\cup \N$, and for~\eqref{eq:expected} we need the event $\left\{\left(x,\{\P_k\}_{k=0}^\infty\right):\ x\in \Omega\ \wedge\  B\left(x,\frac{\beta\diam(X)}{2^k}\right)\subseteq \P_k(x)\right\}$ to be measurable with respect to $\mu\times \Pr$ for all $k\in \{0\}\cup \N$. These assumptions will be trivially satisfied in the concrete constructions below.}
\end{remark}

\begin{remark}\label{rem:arbitrary choices}
{\em In the above definitions we made some arbitrary choices: the factor $\frac{1}{2^k}$ in~\eqref{eq;power 2 decay} can be taken to be some other factor $r_k>0$, and the  $\frac12$ lower bound on the probability in~\eqref{eq:padd} can be taken to be some other probability $p_k$. Since we will not use these additional degrees of freedom here, we chose not to mention them for the sake of simplifying notation. But, the arguments below can be easily carried out in greater generality, which might be useful for future applications of these notions.}
\end{remark}

The following lemma deals with the existence of padded random partition trees on
microdoubling metric measure spaces. The argument is similar to the proof of
Theorem 3.17 in~\cite{LN05}, which is based on ideas from the
theoretical computer science literature~\cite{CKR04,FRT04}. The last
part of the argument is in the spirit of the proof of the main
padding inequality in~\cite{MN07}.

\begin{lemma}\label{lem:decomp} Fix $n\ge 1$ and $K\ge 5$. Let $(X,d,\mu)$ be a separable
 bounded metric measure space which satisfies~\eqref{eq:def micro}.
Then $X$ admits a $\frac{1}{16n\log K}$-padded probability distribution over
partition trees.
\end{lemma}

\begin{remark}\label{rem:volberg}
{\em Let $(X,d)$ is a separable complete and bounded metric space which is doubling with constant $\lambda$, i.e., every ball in $X$ can be covered by at most $\lambda$ balls of half the radius. It is a classical fact, due to Vol{\cprime}berg and
Konyagin~\cite{VK87} in the case of compact spaces, and Luukkainen
and Saksman~\cite{LS98} in the case of general complete spaces (see
also~\cite{Wu98} and chapter 13 in~\cite{Hei01}), that $X$ admits a
non-degenerate measure $\mu$ which is doubling with constant
$\lambda^2$ (the power $2$ can be replaced here by any power bigger than $1$). Thus the conclusion of Lemma~\ref{lem:decomp} holds in this case with $n=1$ and $K=\lambda^2$.
}
\end{remark}

\begin{proof}[Proof of Lemma~\ref{lem:decomp}] By
rescaling the metric we may assume without loss of generality that
$\diam(X)=1$. Since $X$ is bounded, $\mu(X)<\infty$, and we may
therefore normalize $\mu$ to be a probability measure. Let
$x_1,x_2,x_3,\ldots$ be points chosen uniformly and independently at
random from $X$ according to the measure $\mu$, i.e.,
$(x_1,x_2,\ldots)$ is distributed according to the probability
measure $\mu^{\otimes \aleph_0}$. For each $k$ let $r_k$ be a random
variable that is distributed uniformly on the interval
$\left[2^{-k-2},2^{-k-1}\right]$. We assume that
$r_1,r_2,\ldots$ are independent. Let $\Pr$ denote the joint
distribution of $(x_1,x_2,\ldots),(r_1,r_2,\ldots)$.

For every $k\in \N$ define a random variable $j_k:X\to
\N\cup\{\infty\}$ by
$$
j_k(x)\stackrel{\mathrm{def}}{=} \inf\left\{j\in \N\cup\{\infty\}:\ d(x,x_j)\le r_k
\right\}.
$$
Note that $j_k(x)$ is almost surely finite for
every $x\in X$, since each $x_j$ has positive probability of falling
into $B(x,r_k)\supseteq B\left(x,2^{-k-2}\right)$ (see the
argument in~\cite{LN05} for more details). Since $X$ is separable, it
follows that the event $\bigcup_{x\in X}\bigcup_{k=1}^\infty
\{j_k(x)<\infty\}$ has probability $1$. From now on we will
condition on this event.

For every $k\in \N$ and $\ell_1,\ldots,\ell_k\in \N$ define $$
P(\ell_1,\ldots,\ell_k)\stackrel{\mathrm{def}}{=} \left\{x\in X:\
j_1(x)=\ell_1,\ldots,j_k(x)=\ell_k\right\}.$$ Then $\P_k\stackrel{\mathrm{def}}{=}
\{P(\ell_1,\ldots,\ell_k):\ \ell_1,\ldots,\ell_k\in \N\}$ is a
partition of $X$. By definition
$$
P(\ell_1,\ldots,\ell_k)\subseteq B(x_{\ell_k},r_k)\subseteq
B\left(x_{\ell_k},2^{-k-1}\right),
$$
and for all $k\in \N$,
$$
P(\ell_1,\ldots,\ell_k,\ell_{k+1})\subseteq P(\ell_1,\ldots,\ell_k).
$$
Therefore $\P_{k+1}$ is a refinement of $\P_k$ and $\diam
(\P_k(x))\le 2^{-k}$ for all $x\in X$.

Denote
\begin{equation}\label{eq:def beta}
\beta=\frac{1}{16n\log K}\ .
\end{equation}
Since $K\ge 5$, we have $\beta<\frac{1}{25}$. Fix $k\in \N$ and $x\in X$ and observe that
\begin{multline}\label{eq:union}
\Pr\left[B\left(x,\frac{\beta}{2^{k}}\right)\subseteq
\P_k(x)\right]=\Pr\left[\bigcap_{\ell=1}^k \left\{\forall y\in
B\left(x,\frac{\beta}{2^{k}}\right),\
j_\ell(x)=j_\ell(y)\right\}\right]\\\ge 1-\sum_{\ell=1}^k \Pr\left[
\exists y\in B\left(x,\frac{\beta}{2^{k}}\right),\ j_\ell(x)\neq
j_\ell(y)\right].
\end{multline}
Fix $\ell\in \{1,\ldots,k\}$. Note that
\begin{multline}\label{eq:ckr}
\left\{\exists y\in B\left(x,\frac{\beta}{2^{k}}\right),\ j_\ell(x)\neq
j_\ell(y)\right\}\\\subseteq \bigcup_{i=1}^\infty
\bigcap_{j=1}^{i-1} \left\{r_\ell-\frac{\beta}{2^{k}}<d(x_i,x)\le
r_\ell+\frac{\beta}{2^{k}}\ \wedge\  d(x_j,x)>r_\ell+\frac{\beta}{2^{k}}\right\}.
\end{multline}
To prove~\eqref{eq:ckr}, assume that there is some $y\in B\left(x,\frac{\beta}{2^{k}}\right)$ for which $j_\ell(x)\neq j_\ell(y)$. Let $i\in \N$ be the first index such that $d(x_i,x)\le r_\ell+\frac{\beta}{2^{k}}$. Note that in order to prove that the event in the right hand side of~\eqref{eq:ckr} occurs, it suffices to show that the event
$$
\bigcap_{j=1}^{i-1} \left\{r_\ell-\frac{\beta}{2^{k}}<d(x_i,x)\le
r_\ell+\frac{\beta}{2^{k}}\ \wedge\  d(x_j,x)>r_\ell+\frac{\beta}{2^{k}}\right\}$$
occurs, which, by the minimality of $i$, is equivalent to showing that $d(x_i,x)> r_\ell-\frac{\beta}{2^{k}}$. So, assume for the sake of contradiction that
$d(x_i,x)\le r_\ell-\frac{\beta}{2^{k}}$. This implies in particular that $j_\ell(x)=i$, and moreover, since $y\in
B\left(x,\frac{\beta}{2^{k}}\right)$, we have $d(x_i,y)\le r_\ell$, implying that $j_\ell(y)\le i$. But, $d\left(x,x_{j_\ell(y)}\right)\le d\left(y,x_{j_\ell(y)}\right)+d(x,y)\le r_\ell+\frac{\beta}{2^{k}}$, and the minimality  of $i$ implies that $j_\ell(y)\ge i$. Thus $j_\ell(y)=i=j_\ell(x)$, contradicting our assumption on $y$.

Now, \eqref{eq:ckr} implies that
\begin{eqnarray}\label{eq:ratio1}
&&\!\!\!\!\!\!\!\!\!\!\!\!\!\!\!\Pr\left[ \exists y\in B\left(x,\frac{\beta}{2^{k}}\right),\ j_\ell(x)\neq j_\ell(y)\right]\nonumber\\&\le&
2^{\ell+2}\int_{2^{-\ell-2}}^{e^{-\ell-1}} \left(\mu\left(B\left(x,r+\frac{\beta}{2^{k}}\right)\right)-\mu\left(B\left(x,r-\frac{\beta}{2^{k}}\right)\right)\right)
\nonumber\\&\phantom{\le}&\quad\cdot\left(\sum_{i=1}^\infty \nonumber\left(1-\mu\left(B\left(x,r+\frac{\beta}{2^{k}}\right)\right)\right)^{i-1}\right)dr\nonumber\\
&=& 1-2^{\ell +2}\int_{\frac14 e^{-\ell b}}^{\frac12 e^{-\ell b}}
\frac{\mu\left(B\left(x,r-\frac{\beta}{2^{k}}\right)\right)}{
\mu\left(B\left(x,r+\frac{\beta}{2^{k}}\right)\right)}dr,
\end{eqnarray}
 Denote
$h(t)\stackrel{\mathrm{def}}{=} \log \mu\left(B(x,s)\right)$. Then by Jensen's
inequality we see that
\begin{multline}\label{eq:ratio split}
2^{\ell +2}\int_{2^{-\ell -2}}^{2^{-\ell -1}}
\frac{\mu\left(B\left(x,r-\frac{\beta}{2^{k}}\right)\right)}{\mu\left(B\left(x,r-\frac{\beta}{2^{k}}\right)\right)}dr= 2^{\ell +2}\int_{2^{-\ell
-2}}^{2^{-\ell -1}} e^{h\left(r-\frac{\beta}{2^{k}}\right)-h\left(r+\frac{\beta}{2^{k}}\right)}dr\\\ge
\exp\left(2^{\ell +2}\int_{2^{-\ell -2}}^{2^{-\ell-1}}\left[h\left(r-\frac{\beta}{2^{k}}\right)-h\left(r+\frac{\beta}{2^{k}}\right)\right]dr\right).
\end{multline}
The term in the exponent in~\eqref{eq:ratio split} can be estimated as follows:
\begin{multline}\label{eq:ratio split 2}
\int_{2^{-\ell -2}}^{2^{-\ell-1}}\left[h\left(r-\frac{\beta}{2^{k}}\right)-h\left(r+\frac{\beta}{2^{k}}\right)\right]dr=\int_{2^{-\ell -2}-\beta
2^{-k}}^{2^{-\ell -2}+\beta
2^{-k}}h\left(s\right)ds-\int_{2^{-\ell -1}-\beta
2^{-k}}^{2^{-\ell -1}+\beta
2^{-k}}h\left(s\right)ds\\ \ge \beta 2^{-k+1}\left[h\left(2^{-\ell
-2}-\beta e^{-k}\right)-h\left(2^{-\ell -1}+\beta
2^{-k}\right)\right].
\end{multline}
By recalling the definition of $h$, a combination of~\eqref{eq:ratio1}, \eqref{eq:ratio split}, \eqref{eq:ratio split 2} yields the bound,
\begin{equation}\label{eq:prob lower}
\Pr\left[ \exists y\in B\left(x,\frac{\beta}{2^{k}}\right),\ j_\ell(x)\neq j_\ell(y)\right]\ge 1- \left(\frac{\mu\left(B\left(x,2^{-\ell -2}-\beta
2^{-k}\right)\right)}{\mu\left(B\left(x,2^{-\ell -1}+\beta
2^{-k}\right)\right)}\right)^{\beta 2^{-(k-\ell)+3}}.
\end{equation}
Note that since $\ell\le k$ and $\beta\le \frac1{25}$ we know that
$2^{-\ell -1}+\beta
2^{-k}\le \left(1+\frac{1}{n}\right)^{n+1}\left(2^{-\ell -2}-\beta
2^{-k}\right)$. Hence, combining the assumption~\eqref{eq:def micro}
with~\eqref{eq:prob lower}, we see that
\begin{multline}\label{eq:ratio3}
\Pr\left[ \exists y\in B\left(x,\frac{\beta}{2^{k}}\right),\ j_\ell(x)\neq
j_\ell(y)\right] \le 1-K^{-(n+1)\beta
2^{-(k-\ell)+3}}\\\le (n+1)\beta
2^{-(k-\ell)+3}\log K\stackrel{\eqref{eq:def beta}}{\le} 2^{-(k-\ell)}.
\end{multline}
Plugging~\eqref{eq:ratio3} into~\eqref{eq:union} we see that
$$
\Pr\left[B\left(x,\frac{1}{16n\log K}\cdot 2^{-k}\right)\subseteq \P_k(x)\right]=\Pr\left[B\left(x,\frac{\beta}{2^{k}}\right)\subseteq \P_k(x)\right]\ge
1-\sum_{\ell=1}^k 2^{-(k-\ell)}\ge \frac12.
$$
This is precisely the statement that the partition tree
$\{\P_k\}_{k=0}^\infty$ is $\frac{1}{16n\log K}$-padded.
\end{proof}

The connection between the existence of padded random partition
trees and the Hardy-Littlewood maximal inequality is established in
the proof of Theorem~\ref{local}.


\begin{proof}[Proof of Theorem~\ref{local}] By a standard monotone convergence argument we may assume that $R$ is bounded, say
$R\subseteq [0,D]$ for some $D>1$. Fix $f\in L_p(X)$. By homogeneity
it suffices to show that
$$
\mu\left(M_Rf>1\right)\lesssim C^p\left(\left(1+\frac{\log\log K}{1+\log n}\right)Q^p+K^p\right)\int_X|f|^pd\mu,
$$
where $C>0$ is a universal constant and
\begin{equation}\label{eq:def Q}
Q\stackrel{\mathrm{def}}{=}\sup_{r>0}
\left\|M_{R\cap [r,nr]}\right\|_{L_p(X)\to L_{p,\infty}(X)}.
\end{equation}
By monotone convergence we may assume that $f$ (and hence also
$M_Rf$) has bounded support. We would like to apply Theorem
\ref{doob2}, but unfortunately there are no obvious candidates for
$\F_k$ with which we have either \eqref{mnf} or \eqref{emf}.
Nevertheless, we shall be able to proceed by replacing $M_R$ with a
slightly modified variant.

Let $E$ be the support of $f$ and denote $$E'\stackrel{\mathrm{def}}{=} \{x\in X:\
d(x,E)\le D\},$$ and $$E''\stackrel{\mathrm{def}}{=} \{x\in X:\ d(x,E)\le 2D\}.$$ Then
$E\subseteq E'\subseteq E''$ and $\diam(E'')\le 4D+\diam(E)<\infty$.
Moreover the support of $M_Rf$ is contained in $E'$. It will
therefore suffice to prove that
$$
\left\|M_R\right\|_{L_p(E')\to L_{p,\infty}(E'')}\lesssim \left(1+\frac{\log\log K}{1+\log n}\right)Q+K.
$$
By rescaling the metric we may assume that $\diam(E'')=1$. Once this
is achieved we may also assume that $R\subseteq (0,1]$, since the
operator $M_{R\cap (1,\infty)}$, viewed as an operator on $L_p(E')$,
is pointwise bounded by the averaging operator on $E'$.

Using Lemma~\ref{lem:decomp}, let $\left\{\P_k\right\}_{k=0}^\infty$
be a random partition tree on $E''$ which is $\beta$-padded, where
$$
\beta=\frac{1}{16n\log K}.
$$
Let $m$ be the largest integer such that $2^{-m}\le \beta$. Denote
for $k\ge 0$ and $i\in \{1,2,3\}$,
$$
R^i_k\stackrel{\mathrm{def}}{=} R\cap
\left[2^{-(3k+i)m},2^{-(3k-1+i)m}\right]\quad \mathrm{and}\quad
R^i\stackrel{\mathrm{def}}{=} \bigcup_{k\in \N\cup \{0\}} R^i_k.
$$
Thus $R=R^1\cup R^2\cup R^3$, which implies that
\begin{multline}\label{eq:sparsify}
\mu\left(M_Rf>1\right)=\mu\left(\max\left\{
M_{R^1}f,M_{R^2}f,M_{R^3}f\right\}>1\right)\\\le \mu\left(M_{R^1}f>1\right)+\mu\left(M_{R^2}f>1\right)+
\mu\left(M_{R^3}f>1\right).
\end{multline}


Fix $i\in \{1,2,3\}$ and $k\in \N\cup \{0\}$, and define
$$
E_k^i\stackrel{\mathrm{def}}{=} \left\{x\in E':\ M_{R_k^i}f(x)>1\right\}\setminus
\bigcup _{j=0}^{k-1} \left\{x\in E':\ M_{R^i_j}f(x)>1\right\}.
$$
Then the sets $E_k^i$ are disjoint
 and
 \begin{eqnarray}\label{eq:refine}
\mu\left( M_{R^i}f>1\right)=\mu\left( \sup_{k\in
\N\cup\{0\}}M_{R^i_k}f>1\right)=\sum_{k=0}^\infty
\mu\left(E^i_k\right).
 \end{eqnarray}
Recalling~\eqref{eq:def pad}, we denote
$$
\widetilde E_k^i\stackrel{\mathrm{def}}{=}
(E_k^i)_\beta^{\pad((3k+i+1)m)}=\left\{x\in E^i_k:\ B\left(x,\frac{\beta}{
2^{(3k+i+1)m}}\right)\subseteq \P_{(3k+i+1)m}(x)\right\}.
$$
Then by~\eqref{eq:expected} we know that
\begin{eqnarray}\label{eq:expected2}
\E\left[\mu\left(\widetilde E_k^i\right)\right]\ge \frac{ \mu\left(
E_k^i\right)}{2}.
\end{eqnarray}
Plugging~\eqref{eq:expected2} into~\eqref{eq:refine} we see that
\begin{eqnarray}\label{eq:bound by expectation}
\mu\left(M_{R^i}f>1\right)\le 2\E \left[\sum_{k=0}^\infty
\mu\left(\widetilde E^i_k\right)\right]=2\E \left[\mu\left(
\sup_{k\in \N\cup\{0\}} \widetilde M_{R^i_k}f>1\right)\right],
\end{eqnarray}
where $\widetilde M_{R^i_k}$ is the sublinear operator
$$
\widetilde M_{R^i_k}g\stackrel{\mathrm{def}}{=} \1_{\widetilde E_k^{i}}M_{R^i_k}g.
$$

Write $r=2^{-(3k+i)m}$ and let $v\asymp 1+\frac{\log \log K}{1+\log n}$ be an integer such that $2^{m/v}\le n$. By the definition of $Q$, for every $g\in L_p(E')$ and $t>0$ we have
\begin{multline*}
\mu\left(\widetilde M_{R^i_k}g>t\right)\le \mu\left(M_{R^i_k}g>t\right)= \mu\left(M_{R\cap [r,2^mr]}g>t\right)\\
\le \sum_{u=0}^{v-1} \mu\left(M_{R\cap \left[r2^{\frac{um}{v}},nr2^{\frac{um}{v}}\right]}g>t\right)
\le vQ^p\frac{\|g\|_{L_p(E')}^p}{t^p}.
\end{multline*}
Thus,
\begin{eqnarray}\label{eq:first condition}
 g\in L_p(E')\implies \left\| \widetilde M_{R_k^i} g
\right\|_{L_{p,\infty}(E')} \leq v^{1/p}Q \| g\|_{L_p(E')}.
\end{eqnarray}

For every $k\in \N\cup\{0\}$ we let $\F_k\stackrel{\mathrm{def}}{=}
\sigma (\P_k)$ be the $\sigma$-algebra generated by the partition
$\P_k$. Then $\F_0\subseteq \F_1\subseteq\F_2\subseteq  \cdots$. We
claim that for every $k\in \N\cup\{0\}$, if $F\in \F_{(3k+i+1)m}$
then
\begin{eqnarray}\label{eq:goal local}
\1_{F}\widetilde M_{R_{k+1}^i}(g) =\widetilde
M_{R_{k+1}^i}(\1_{F}g).
\end{eqnarray}
By the definition of $\widetilde M_{R_k^i}$, in order to
prove~\eqref{eq:goal local} we have to show that for almost every
$x\in E'$ we have
\begin{eqnarray}\label{eq:must show identity}
\1_{F}(x)\cdot\1_{\widetilde E_{k+1}^{i}}(x)\cdot
M_{R_{k+1}^i}(g)(x) =\1_{\widetilde E_{k+1}^{i}}(x)\cdot
M_{R_{k+1}^i}(\1_{F}g)(x).
\end{eqnarray}
It is non-trivial to check~\eqref{eq:must show identity} only  when
$x\in \widetilde E_{k+1}^{i}$, in which case we are guaranteed that
$ B\left(x,\beta 2^{-(3k+i+1)m}\right)\subseteq \P_{(3k+i+1)m}(x)$.
But since $F\in \F_{(3k+i+1)m}$, we know that $P_{(3k+i+1)m}(x)$ is
either disjoint from $F$ or contained in $F$. If
$P_{(3k+i+1)m}(x)\subseteq F$, then for every $r\in R_{k+1}^i$,
\begin{eqnarray}\label{eq:implying}
B(x,r)\subseteq B\left(x,2^{-(3k+i+2)m}\right)\subseteq
B\left(x,\beta 2^{-(3k+i+1)m}\right)\subseteq
\P_{(3k+i+1)m}(x)\subseteq F,
\end{eqnarray}
where we used the fact that $r\le 2^{-(3(k+1)-1+i)m}$ and $2^{-m}\le \beta$. The
inclusion~\eqref{eq:implying} implies that both sides of the
equation~\eqref{eq:must show identity} are equal to
$M_{R_{k+1}^i}(g)(x)$. On the other hand, if $P_{(3k+i+1)m}(x)$ is
disjoint from $F$, then $B(x,r)$ is disjoint from $F$ for all $r\in
R_{k+1}^i$, implying that both sides of the equation~\eqref{eq:must
show identity} vanish. This concludes the proof of~\eqref{eq:goal
local}.

Fix $g\in L_\infty(E')$, and extend $g$ to a function on $X$ whose value
is $0$ outside $E'$. Assume that $$\left\|\E\left(|g|\big|
\F_{(3k+i+1)m}\right)\right\|_{L_\infty(E')}=1.$$ This implies that
for all $F\in \F_{(3k+i+1)m}$ we have
\begin{eqnarray}\label{eq:qeak bound}
\int_{F}|g|d\mu=\int_{F\cap E'}|g|d\mu\le \mu\left(F\cap
E'\right)\le \mu(F).
\end{eqnarray}
Fix $r\in R_k^i$ and $x\in E'$. Denote $$
F\stackrel{\mathrm{def}}{=} \bigcup\left\{C\in \P_{(3k+i+1)m}:\
C\cap B(x,r)\neq \emptyset\right\}\in \F_{(3k+i+1)m}.$$ Note that
$B(x,r)\subseteq E''$, which implies that
\begin{equation}\label{eq:F contains ball}
F\supseteq B(x,r). \end{equation}
Moreover,
\begin{eqnarray}\label{eq:inc}
F\subseteq B\left(x,r+\sup_{C\in
\P_{(3k+i+1)m}}\diam(C)\right)\subseteq
B\left(x,r+2^{-(3k+i+1)m}\right)\subseteq
B\left(x,\left(1+2^{-m}\right)r\right),
\end{eqnarray}
where in the last inclusion in~\eqref{eq:inc} we used the fact that
$r\in R_k^i$ implies that $r\ge 2^{-(3k+i)m}$. Hence,
\begin{multline}\label{eq:second condition}
\frac{1}{\mu(B(x,r))}\int_{B(x,r)} |g|d\mu\stackrel{\eqref{eq:F
contains ball}}{\le} \frac{1}{\mu(B(x,r))}\int_{F}
|g|d\mu\stackrel{\eqref{eq:qeak bound}}{\le}
\frac{\mu(F)}{\mu(B(x,r))}\\\stackrel{\eqref{eq:inc}}{\le}
\frac{\mu\left(B\left(x,\left(1+2^{-m}\right)r\right)\right)}{\mu(B(x,r))}\le
\frac{\mu\left(B\left(x,\left(1+\frac{1}{n}\right)r\right)\right)}{\mu(B(x,r))}
\stackrel{\eqref{eq:def micro}}{\le} K,
\end{multline}

We are now in position to apply Theorem~\ref{doob2} to the increasing sequence of $\sigma$-algebras $\left\{\F_{(3k+i+1)m}\right\}_{k=0}^\infty$ and the sublinear operators $\left\{M_{R^i_k}\right\}_{k=0}^\infty$, with $A=v^{1/p}Q$, due to~\eqref{eq:first condition}, and $B=K$, due to~\eqref{eq:second condition}:
\begin{multline*}\
\mu\left( \sup_{k\in \N\cup\{0\}} \widetilde M_{R^i_k}f>1\right)\le \left(2^pvQ^p+2^pK^p\right)\int_X|f|^pd\mu\\\lesssim \left(2^p\left(1+\frac{\log\log K}{1+\log n}\right)Q^p+2^pK^p\right)\int_X|f|^pd\mu.
\end{multline*}
Using~\eqref{eq:bound by expectation} and~\eqref{eq:sparsify}, we therefore deduce that
$$
\left[\mu\left(M_Rf>1\right)\right]^{1/p}\lesssim \left(\left(1+\frac{\log\log K}{1+\log n}\right)^{1/p}Q+K\right)\|f\|_{L_p(X)},
$$
as required.
\end{proof}

\section{An argument of E. Lindenstrauss}\label{lin-sec}

We now present an alternative approach to Corollary~\ref{coro:our norm bounds},
following an argument of E. Lindenstrauss~\cite{linden}. Let us first make some definitions. We fix a metric measure space
$(X,d,\mu)$. Given any two radii $r,r' > 0$ and a center $x \in X$,
we define the enlarged ball $B(x,r,r')$ by
$$ B(x,r,r') \stackrel{\mathrm{def}}{=} \bigcup_{y \in B(x,r)} B(y,r') =
\{ z \in X: d(x,y) \le r\ \wedge\  d(y,z) \le r' \hbox{ for some } y \in
X \}.$$ Thus, for instance,
\begin{equation}\label{eq:inclusion extended ball}
B(x,r) \subseteq B(x,r,r') \subseteq
B(x,r+r').
\end{equation}
In analogy to~\cite{linden}, we say that a finite
sequence of radii $0 < r_1 < r_2 < \cdots < r_k$ is \emph{tempered}
with constant $K \geq 1$ if we have the bound
\begin{multline}\label{eq:def tempered}
 \forall\ j\in \{1,\ldots,k\}\ \forall x\in X\ \forall y \in B(x,r_j),\\\mu\left(B(x,r_j) \bigcup \left( \bigcup_{i =1}^{j-1} B\left(x,r_j,r_i\right) \right)\right) \leq K \mu\left( B\left(y,r_j\right) \right).
 \end{multline}


\begin{theorem}[Lindenstrauss maximal inequality]\label{lmi}  Let $(X,d,\mu)$ be a metric measure space, and let $0 < r_1 < r_2 < \ldots < r_k$ be a sequence of radii which is tempered with constant $K$.  Then we have the weak $(1,1)$ maximal inequality
$$ \mu\left( x \in X: \max_{1 \leq j \leq k} \frac{1}{B(x,r_j)} \int_{B(x,r_j)} |f|\ d\mu > \lambda
\right) \leq \frac{2e}{e-1} \frac{K}{\lambda} \|f\|_{L_1(X)}$$ for
all $f \in L_1(X)$ and $\lambda > 0$.
\end{theorem}

\begin{proof}[Proof of Corollary~\ref{coro:our norm bounds} assuming Theorem \ref{lmi}]
Assume that $(X,d,\mu)$ obeys the strong microdoubling condition~\eqref{eq:def strong micro}. It is immediate to
check that any sequence $0 < r_1 < r_2 < \ldots < r_k$ obeying the
lacunarity condition $r_j \geq n r_{j-1}$ will be tempered with
constant $K$, and hence by Theorem~\ref{lmi},
$$ \mu\left(  x \in X: \max_{1 \leq j \leq k} \frac{1}{B(x,r_j)} \int_{B(x,r_j)} |f|\ d\mu > \lambda
\right ) \le \frac{2e}{e-1}\frac{K}{\lambda} \|f\|_{L_1(X)}.$$ If
instead we have the lacunarity condition $r_j \geq 2 r_{j-1}$, then
we can sparsify this sequence into $O(\log n)$ subsequences obeying
the prior lacunarity condition, and hence, by subadditivity,
$$ \mu\left( x \in X: \max_{1 \leq j \leq k} \frac{1}{B(x,r_j)} \int_{B(x,r_j)} |f|\ d\mu > \lambda
\right ) \lesssim \frac{K \log n}{\lambda} \|f\|_{L_1(X)}.$$ From
monotone convergence we then conclude~\eqref{eq:our lacunary}.  Similarly,
any sequence obeying the lacunarity condition $r_j \geq (1 +
\frac{1}{n}) r_{j-1}$ can be sparsified into $O( n \log n )$
sequences which have a lacunarity ratio of $n$.  By monotone
convergence this implies that
$$ \mu\left( x \in X: \sup_{r \in (1 + \frac{1}{n})^\Z} \frac{1}{B(x,r)} \int_{B(x,r)} |f|\ d\mu > \lambda
\right) \lesssim \frac{K n \log n}{\lambda} \|f\|_{L_1(X)},$$ where
$(1 + \frac{1}{n})^\Z$ denotes the integer powers of $1 +
\frac{1}{n}$. Now note from \eqref{eq:def strong micro} that every ball is
contained in a ball whose radius is an integer power of $1 +
\frac{1}{n}$, and whose measure is at most $K$ times larger.  Thus
$$ M f(x) \leq K \sup_{r \in (1 + \frac{1}{n})^\Z} \frac{1}{B(x,r)} \int_{B(x,r)} |f|\ d\mu, $$
and~\eqref{eq:our SS} follows.
\end{proof}

\begin{proof}[Proof of Theorem~\ref{lmi}]
As in \cite{linden}, this is achieved by a randomized variant of the
Vitali covering argument. We may take $f$ to be non-negative, and
normalize $\lambda = 1$. For each $j\in \{1,\ldots,k\}$, let $E_j$ be a
compact subset of $X$ on which we have
\begin{equation}\label{brj}
x \in E_j\implies \frac{1}{B(x,r_j)} \int_{B(x,r_j)} f\ d\mu > 1.
\end{equation}
By inner regularity it will suffice to show that
\begin{equation}\label{eq:for induction elon}
\mu\left( \bigcup_{j=1}^k E_j \right) \leq \frac{2e}{e-1} K \int_X f\ d\mu.
\end{equation}
We establish~\eqref{eq:for induction elon} by induction on $k$.  The case $k=0$ is vacuously
true, so suppose $k \geq 1$ and the claim has already been proven
for $k-1$ (i.e, that~\eqref{eq:for induction elon} holds true for all non-negative $f\in L_1(X)$ and all sets $\{E_j\}_{j=1}^{k-1}$ satisfying~\eqref{brj}).

By compactness, we see that there exists an $\eps > 0$
such that
$$x\in E_k\implies \mu( B(x,r_k) ) > \eps.$$
We then define the extended ball
$$ B^*(x) \stackrel{\mathrm{def}}{=} B(x,r_k) \bigcup \left(\bigcup_{j =1}^{k-1} B(x,r_k,r_j)\right).$$
Thus, since the sequence of radii $\{r_j\}_{j=1}^k$ is tempered, for all $y \in B(x,r_k)$,
\begin{equation}\label{bdouble}
 \eps < \mu\left( B^*(y) \right) \leq K \mu( B(x,r_k) ).
 \end{equation}
 If we then define the \emph{intensity
function}
$$ p(x) \stackrel{\mathrm{def}}{=} \inf_{y\in B(x,r_k)}\frac{1}{\mu( B^*(y) )},$$
then $p$ is a measurable function on $E_k$ which is bounded both
above and below:
\begin{equation}\label{eq:bounds for p}
\frac{1}{K\mu\left(\left(B(x,r_k\right)\right)}\le p(x)<\frac{1}{\e}.
\end{equation}

 We now introduce a Poisson process $\Sigma$ on
$E_k$ with intensity $p(x)$.  Thus $\Sigma$ is a random finite
subset\footnote{If $E_k$ contains atoms, then $\Sigma$ may contain
multiplicity, thus it is really a multiset rather than a set in this
case. One way to create $\Sigma$ is to let $N$ be a Poisson random
variable with expectation $P \stackrel{\mathrm{def}}{=} \int_{E_k} p d\mu$ and then
let $\Sigma = \{x_1,\ldots,x_N\}$ where $x_1,\ldots,x_N$ are iid
elements of $E$ chosen using the probability distribution $p
d\mu_Y/P$.} of $E_k$ which will be almost surely finite, and more
precisely, for any non-negative measurable weight $w: E_k \to \R_+$,
the quantity $\sum_{x \in \Sigma} w(x)$ is a Poisson random variable
with expectation
\begin{equation}\label{eq:def alpha w}
\alpha_w \stackrel{\mathrm{def}}{=}\E \left[\sum_{x \in \Sigma} w(x)\right] =
\int_{E_k} w p\ d\mu,
\end{equation}
i.e., for any integer $k \geq 0$ \begin{eqnarray}\label{expect}
\Pr\left( \sum_{x \in \Sigma} w(x) = k \right) = \frac{e^{-\alpha_w}
\alpha_w^k}{k!}.\end{eqnarray}

Now we define the random sets
$$ E' \stackrel{\mathrm{def}}{=} \bigcup_{x \in \Sigma} B^*(x)\quad \mathrm{and}\quad
F \stackrel{\mathrm{def}}{=} \bigcup_{x \in \Sigma} B(x,r_k).
$$
Then, \begin{eqnarray}\label{eq:three term} \mu\left(
\bigcup_{j=1}^k E_j \right) \leq \mu(E_k) + \mu( E' ) + \mu\left(
\bigcup_{j=1}^{k-1} E_j \setminus E' \right).\end{eqnarray} Let us
investigate the third term in~\eqref{eq:three term}.  Fix $j\in \{1,\ldots,k-1\}$. If $x \in E_j
\setminus E'$, then
$$ \frac{1}{B(x,r_j)} \int_{B(x,r_j)} f\ d\mu > 1.$$
But, since $x\notin E'$ it follows from our definitions that $B(x,r_j)$ is disjoint from $F$.  Thus we have
$$ \frac{1}{B(x,r_j)} \int_{B(x,r_j)} f \1_{X \setminus F} \ d\mu > 1.$$
We can therefore apply the induction hypothesis to the sets $\{E_j\setminus E'\}_{j=1}^{k-1}$ and the function $f\1_{X\setminus F}$, and conclude that
$$ \mu\left( \bigcup_{j=1}^{k-1} E_j \setminus E' \right) \leq \frac{2e}{e-1} K \int_{X \setminus F} f\ d\mu.$$
It follows from~\eqref{eq:three term} that it suffices to show that
\begin{equation}\label{eq:goal after expectation}
\mu(E_k) + \E \left[\mu(E')\right]\le \E \left[\mu(E_k) + \mu(E') \right]
\leq \frac{2e}{e-1} K\E\left[ \int_{F} f\ d\mu\right].
\end{equation}
Now, applying~\eqref{eq:def alpha w} and
\eqref{expect} with $w \stackrel{\mathrm{def}}{=} 1/p$, we have
$$ \mu(E_k) = \E \left[\sum_{x \in \Sigma} \frac{1}{p(x)}\right],$$
while from definition of $E'$ we have
$$ \mu(E') \leq \sum_{x \in \Sigma} \mu\left( B^*(x) \right) = \sum_{x \in \Sigma} \frac{1}{p(x)}.$$
Thus, in order to prove~\eqref{eq:goal after expectation} it suffices to show that
\begin{equation}\label{eq:new goal after expectation} \E \left[\sum_{x \in \Sigma} \frac{1}{p(x)}\right] \leq \frac{e}{e-1} K \E \left[\int_F f\ d\mu\right].\end{equation}

>From \eqref{brj} we know that for all $x \in \Sigma$,
$$ \frac{1}{p(x)}< \frac{1}{p(x) \mu(B(x,r_k))} \int_X \1_{B(x,r_k)} f\
d\mu,
$$
 and hence
\begin{equation}\label{eq:first step in new goal} \E \left[\sum_{x \in \Sigma} \frac{1}{p(x)}\right]
\leq \int_X \left(\E \left[\sum_{x \in \Sigma} \frac{1}{p(x)
\mu(B(x,r_k))} \1_{B(x,r_k)}\right]\right) f\ d\mu.\end{equation}
Fix $y\in X$. From~\eqref{eq:def alpha w} with $w(x)=\frac{\1_{B(x,r_k)}(y)}{p(x)\mu\left(B(x,r_k)\right)}$, we see that
\begin{equation}\label{eq:second step in new goal} \E \left[\sum_{x \in \Sigma} \frac{1}{p(x) \mu(B(x,r_k))}\1_{B(x,r_k)}(y)\right] = \int_{E_k \cap B(y,r_k)}
\frac{1}{\mu(B(x,r_k))} \ d\mu(x).\end{equation}
By substituting~\eqref{eq:second step in new goal} into~\eqref{eq:first step in new goal}, we see that in order to prove~\eqref{eq:new goal after expectation} it will suffice to prove
the pointwise estimate
\begin{equation}\label{eq:new new goal} \int_{E_k \cap B(y,r_k)} \frac{1}{\mu(B(x,r_k))} \ d\mu(x)
\leq \frac{eK}{e-1} \E\left[ \1_F(y)\right],
\end{equation} for all $y\in X$.

Now observe that the definition of $F$ implies that $\1_F(y) = 1$ if and only if $|\Sigma \cap
B(y,r_k)| \geq 1$.  But, recall from~\eqref{eq:def alpha w} (using $w(x)=\1_{B(y,r_k)}(x)$) that $|\Sigma \cap B(y,r_k)|$ is a Poisson random
variable with expectation
\begin{equation}\label{eq:def ay} \alpha(y) \stackrel{\mathrm{def}}{=} \int_{E_k} \1_{B(y,r_k)} p\ d\mu = \int_{E_k \cap B(y,r_k)} p(x)\ d\mu(x),\end{equation}
and thus
\begin{equation}\label{eq:expectation alphay} \E \left[\1_F(y)\right] = 1 - e^{-\alpha(y)}.\end{equation}
A combination of~\eqref{eq:bounds for p} and~\eqref{eq:def ay} yields the bound
\begin{equation}\label{eq:use tempered}
\int_{E_k \cap B(y,r_k)} \frac{1}{\mu(B(x,r_k))} \ d\mu(x) \leq K\alpha(y).\end{equation}
The definition of $p(x)$ implies that if $y\in B(x,r_k)$ then $p(x)\le \frac{1}{\mu(B^*(y))}\le \frac{1}{\mu(B(y,r_k))}$, since $B^*(y)\supseteq B(x,r_k)$. In combination with~\eqref{eq:def ay}, we deduce that $\alpha(y)\le 1$. But, the function $\alpha\mapsto \frac{1-e^{-\alpha}}{\alpha}$ is decreasing on $[0,\infty)$, and therefore $1-e^{-\alpha(y)}\ge (1-e^{-1})\alpha(y)$. This, in combination with~\eqref{eq:expectation alphay} and~\eqref{eq:use tempered}, implies~\eqref{eq:new new goal}, and completes the proof of Theorem~\ref{lmi}.
\end{proof}

As observed in \cite{linden}, the above argument allows us to
extract a good maximal inequality for sufficiently sparse
subsequences of radii if the situation is sufficiently ``amenable''.
In our current context, the analogue for amenability is in fact
subexponential growth:

\begin{corollary} Let $(X,d,\mu)$ be a metric measure space such that $\mu(B(x,r)$ is independent of $x\in X$ for all $r>0$.
Suppose also that we have the sub-exponential growth condition
\begin{equation}\label{subexp}
\lim_{r \to \infty} \frac{\log \mu(B(x,r))}{r} = 0
\end{equation}
for any $x \in X$ (note that our assumption implies that the choice of $x$
is in fact irrelevant).  Then there exists a sequence of radii $0 <
r_1 < r_2 < \ldots$ tending to infinity such that we have the
maximal inequality
$$ f\in L_1(X)\implies \left\| \sup_{k \geq 1} A_{r_k} |f| \right\|_{L_{1,\infty}(X)} \leq 4 \|f\|_{L_1(X)},$$
where the averaging operators $A_{r}$ are given by $A_{r}g\stackrel{\mathrm{def}}{=}
\frac{1}{B(x,r)} \int_{B(x,r)} |g|\ d\mu.$
\end{corollary}

\begin{proof}  We construct the radii recursively as follows.  We set $r_1 \stackrel{\mathrm{def}}{=} 1$.
If $r_1,\ldots,r_k$ have already been chosen, we choose $r_{k+1} >
\max\left\{r_k,k\right\}$ so that
$$ \log \mu\left( B\left(x, r_{k+1} + r_k \right) \right) \leq \mu\left(B\left(x, r_{k+1} \right)\right) + 0.001$$
for any $x \in X$. Such a radius must exist, since otherwise one
would easily contradict~\eqref{subexp}.  The sequence of radii is
tempered with constant $K = e^{0.001}$, and the claim follows since
$\frac{2K}{1-e^{-1}} < 4$.
\end{proof}

\section{The infinite tree}\label{nevo-sec}
 Fix $k\ge 2$ and let $T$ be the infinite rooted $k$-ary tree with
 the usual graph metric and the counting measure $\mu$.
In this section we prove Theorem~\ref{free}.  The first (standard)
step is to replace the Hardy-Littlewood maximal function with the
spherical maximal function
$$ M^\circ f(x) \stackrel{\mathrm{def}}{=} \sup_{r \geq 0} \frac{1}{|S(x,r)|} \sum_{y \in S(x,r)} |f(y)|,$$
where  $S(x,r)$ is the sphere
$$ S(x,r) \stackrel{\mathrm{def}}{=} \{ y \in T: d(x,y) = r \}.$$
Since every ball can be written as the disjoint union of spheres, we
have the pointwise estimate
$$ Mf(x) \leq M^\circ f(x),$$
and so it suffices to show that
\begin{equation}\label{treeweak}
\mu\left( x \in T: M^\circ f(x) \geq \lambda \right) \lesssim
\frac{1}{\lambda} \| f \|_{L_1(T)},
\end{equation}
for all $f \in L_1(T)$ and $\lambda > 0$.


Our arguments rely on the following expander-type estimate.  We use
$|E| = \mu(E)$ to denote the cardinality of a finite set $E\subseteq
T$.

\begin{lemma}\label{lxy}  Let $E, F$ be finite subsets of $T$ and let $r \geq 0$ be an integer.  Then
$$ |\{ (x,y) \in E \times F:\  d(x,y) = r \}| \le 2|E|^{1/2} |F|^{1/2} k^{r/2}.$$
\end{lemma}

This bound should be compared against the ``trivial'' bounds of $|E|
k^r$ and $|F| k^r$. It is superior when $|E|/|F|$ lies between $k^r$
and $k^{-r}$.   By setting $E$ and $F$ equal to concentric spheres
one can verify that the bound is essentially sharp in this case.

\begin{proof}  Let us subdivide $T = \bigcup_{j = 0}^\infty T_j$, where $T_j$
is the generation of the tree at depth $j$ (thus for instance $|T_j|
= k^j$). We then define $E_j \stackrel{\mathrm{def}}{=} E \cap T_j$
and $F_j \stackrel{\mathrm{def}}{=} F \cap T_j$. Observe that in
order for an element in $E_j$ and an element in $F_{i}$ to have
distance exactly $r$, we must have $i = j + r - 2m$ for some $m\in
\{0,\ldots,r\}$. Thus we can write
\begin{equation}\label{eq:split levels} |\{ (x,y) \in E \times F:
d(x,y) = r \}| = \sum_{m=0}^r \sum_{\substack{i,j \in \N\cup\{0\}\\
i = j+r-2m}} \left| \{ (x,y) \in E_j \times F_{i}: d(x,y) = r
\}\right|.\end{equation}

Fix $m\in \{0,\ldots,r\}$ and $i,j\in \N\cup\{0\}$ such that
$i=j+r-2m$. Observe that if $x \in T_j$ and $y \in T_{i}$ are at
distance $r$ in $T$, then the $m^{th}$ parent of $x$ equals the
$(r-m)^{th}$ parent of $y$. From this we conclude that for each $x
\in T_j$ there are at most $k^{r-m}$ elements of $y \in T_{i}$ with
$d(x,y) = r$, and conversely for each $y \in T_{i}$ there are at
most $k^m$ elements of $x \in T_j$ with $d(x,y) = r$. Thus
\begin{equation}\label{mth} \left| \{ (x,y) \in E_j \times F_{i}: d(x,y) = r
\}\right| \leq \min\left\{ k^{r-m} |E_j|, k^m |F_{i}|
\right\}.\end{equation} A combination of~\eqref{eq:split levels}
and~\eqref{mth} implies that our task is therefore to show that
\begin{equation}\label{eq:tree goal} \sum_{m=0}^r \sum_{\substack{i,j \in \N\cup\{0\}\\
i = j+r-2m}} \min\left\{ k^{r-m} |E_j|, k^m |F_{i}| \right\} \le
2|E|^{1/2} |F|^{1/2} k^{r/2}.\end{equation} If we write $c_j
\stackrel{\mathrm{def}}{=} \frac{|E_j|}{k^j}$ and $d_j
\stackrel{\mathrm{def}}{=} \frac{|F_j|}{k^j}$ for $j \geq 0$ and
$c_j \stackrel{\mathrm{def}}{=} d_j \stackrel{\mathrm{def}}{=} 0$
for $j < 0$ then we have \begin{equation}\label{eq:EF split}
\sum_{j=0}^\infty k^j c_j = |E| \quad\mathrm{and}\quad
\sum_{j=0}^\infty k^j d_j = |F|,\end{equation} and we have
\begin{multline}\label{eq:drop condition}
\sum_{m=0}^r \sum_{\substack{i,j \in \N\cup\{0\}\\
i = j+r-2m}} \min\left\{ k^{r-m} |E_j|, k^m |F_{i}| \right\} =
\sum_{m=0}^r \sum_{\substack{i,j \in \N\cup\{0\}\\
i = j+r-2m}} k^{(i+j+r)/2} \min\left\{ c_j, d_{i} \right\} \\ \leq
k^{r/2} \sum_{i,j=0}^\infty k^{(i+j)/2} \min\left\{c_j, d_{i}
\right\}.
\end{multline}
A combination of~\eqref{eq:EF split} and~\eqref{eq:drop condition}
shows that in order to prove~\eqref{eq:tree goal} it will suffice to
show that
$$ \sum_{i,j=0}^\infty k^{(i+j)/2} \min\left\{c_j, d_{i} \right\}
\le 2\left(\sum_{j\ge 0} k^j c_j\right)^{1/2} \left(\sum_{i\ge 0}
k^{i} d_{i}\right)^{1/2}.$$ To prove this inequality, let $\alpha$
be a real parameter to be chosen later, and estimate
\begin{eqnarray*}
 \sum_{i,j=0}^\infty k^{(i+j)/2} \min\left\{c_j, d_{i} \right\}
 \leq \sum_{\substack{i,j\in \N\cup\{0\}\\ i < j+\alpha}} k^{(i+j)/2} c_j + \sum_{\substack{i,j\in \N\cup \{0\}\\ i \geq j+\alpha}}
 k^{(i+j)/2} d_{i}
 \le \sum_{j=0}^\infty k^{j + \frac{\alpha}{2}} c_j + \sum_{i=0}^\infty k^{i -\frac{\alpha}{2}} d_{i}.
\end{eqnarray*}
Optimising in $\alpha$ we obtain the required result.
\end{proof}


For each $r \geq 0$, let $A^\circ_r$ denote the spherical averaging
operator
$$ A^\circ_r f(x) \stackrel{\mathrm{def}}{=} \frac{1}{\mu(S(x,r))} \sum_{y \in S(x,r)} |f(y)|.$$
Thus $M^\circ f(x)=\sup_{r\ge0} A^\circ_r f(x)$. We can use Lemma
\ref{lxy} to obtain a distributional estimate on $A^\circ_r$.

\begin{lemma}\label{lem:dist}  Let $f \in L_1(T)$, $r > 0$ and $\lambda > 0$.  Then
$$ \mu\left(A^\circ_r f \geq \lambda\right) \lesssim \sum_{\substack{n\in \N\cup\{0\}\\ 1 \leq 2^n \leq 2k^r}}
\sqrt{\frac{2^n}{k^r}}\cdot 2^n \mu\left( |f| \geq 2^{n-1} \lambda
\right).$$
\end{lemma}

\begin{proof}  We may take $f$ to be non-negative.  By dividing $f$ by $\lambda$ we may normalize $\lambda=1$.
We bound \begin{equation}\label{eq:spherical split}
f \leq \frac{1}{2} + \sum_{\substack{n\in \N\cup\{0\}\\
1 \leq 2^n \leq k^r}} 2^n \1_{E_n} + f \1_{\{f \geq \frac12
k^r\}},\end{equation} where $E_n$ is the sublevel set
\begin{equation}\label{eq:def E_n tree} E_n \stackrel{\mathrm{def}}{=} \left\{ 2^{n-1} \leq
f < 2^n\right\}.\end{equation} Hence
\begin{equation}\label{eq:split after average}A^\circ_r f \leq \frac{1}{2} + \sum_{\substack{n\in \N\cup\{0\}\\
1 \leq 2^n \leq k^r}} 2^n A^\circ_r \left(\1_{E_n}\right) +
A^\circ_r \left(f \1_{\{f \geq \frac12 k^r\}}\right).\end{equation}
Since $\mu(S(x,r)) \le k^r$ we see that
\begin{equation}\label{eq:trivial bound in tree} \mu\left(A^\circ_r
\left(f \1_{\{f \geq \frac12 k^r\}}\right) \neq 0 \right) \le k^r
\mu\left( f \geq \frac12 k^r \right).\end{equation} Thus we have
$$
\mu\left( A^\circ_r f \geq 1 \right) \stackrel{\eqref{eq:split after
average}\wedge\eqref{eq:trivial bound in tree}}{\leq} \mu\left(
\sum_{\substack{n\in \N\cup\{0\}\\
1 \leq 2^n \leq k^r}} 2^n A^\circ_r \left(\1_{E_n}\right) \geq
\frac{1}{2} \right) + k^r \mu\left( f \geq \frac12 k^r \right).$$

Note that if
$$ \sum_{\substack{n\in \N\cup\{0\}\\
1 \leq 2^n \leq k^r}} 2^n A^\circ_r \left(\1_{E_n}\right) \geq \frac{1}{2}$$
then we necessarily have for some $n\in \N$ such that $1 \leq 2^n
\leq k^r$,
$$ A^\circ_r \left(\1_{E_n}\right) \geq \frac{1}{2^{n+4}}\left(\frac{2^n}{k^r}\right)^{1/4}.$$
Indeed, otherwise we have
$$
\frac12\le  \sum_{\substack{n\in \N\cup\{0\}\\
1 \leq 2^n \leq k^r}} 2^n A^\circ_r\left( \1_{E_n}\right)\le
\frac1{16}\sum_{\substack{n\in \N\cup\{0\}\\
1 \leq 2^n \leq k^r}}
\left(\frac{2^n}{k^r}\right)^{1/4} \le
\frac{2^{1/4}k^{r/4}-1}{16k^{r/4}\left(2^{1/4}-1\right)}<\frac12,
$$
which is a contraction.  Thus
\begin{eqnarray}\label{eq:break}
\mu\left( A^\circ_r f \geq 1 \right) \leq \sum_{\substack{n\in \N\cup\{0\}\\
1 \leq 2^n \leq k^r}} \mu(F_n) +k^r \mu\left( f \geq \frac12 k^r
\right),\end{eqnarray} where
$$ F_n \stackrel{\mathrm{def}}{=} \left\{ A^\circ_r \left(\1_{E_n}\right) \geq \frac{1}{2^{n+4}}\left(\frac{2^n}{k^r}\right)^{1/4} \right\}.$$
Note that $F_n$ is finite and observe that
$$\frac{1}{k^r} \left|\{ (x,y) \in E_n \times F_n: d(x,y) = r \}\right|=\sum_{y \in
F_n} A^\circ_r \left(\1_{E_n}\right)(y)\ge
\frac{\mu(F_n)}{2^{n+4}}\left(\frac{2^n}{k^r}\right)^{1/4}.$$
Applying Lemma \ref{lxy} we conclude that
$$ \frac{\mu(F_n)}{2^{n+4}}\left(\frac{2^n}{k^r}\right)^{1/4} \le 2\sqrt{\frac{\mu(E_n)
\mu(F_n)}{k^r}}.$$ Hence
$$ \mu(F_n) \le 2^{10}\sqrt{\frac{2^n}{k^r}}\cdot 2^n \mu(E_n).$$
Plugging this estimate into~\eqref{eq:break}, we obtain the required
result.
\end{proof}

\begin{proof}[Proof of Theorem~\ref{free}] Now we prove~\eqref{treeweak}.  Since
$M^\circ f = \sup_{r \geq 0} A^\circ_r f$,
Lemma~\ref{lem:dist} implies that
\begin{multline*}
\mu\left( M^\circ f \geq \lambda\right)  \leq \sum_{r = 0}^\infty
\mu\left( A^\circ_r f \geq \lambda \right)   \lesssim \sum_{r
=0}^\infty \sum_{\substack{n\in \N\cup\{0\}\\
1 \leq 2^n \leq 2k^r}} \sqrt{\frac{2^n}{k^r}}\cdot
2^n \mu\left( |f| \geq 2^{n-1} \lambda
\right)\\
= \sum_{x \in T} \sum_{n =0}^\infty \left(\sum_{\substack{r\in \N\cup \{0\}\\ k^r \geq 2^{n-1}}}
\frac{1}{k^{r/2}}\right) 2^{3n/2} \1_{\{|f(x)| \geq 2^{n-1}
\lambda\}} \lesssim \sum_{x \in T} \sum_{n =0}^\infty 2^n
\1_{\{|f(x)| \geq 2^{n-1} \lambda\}}  \lesssim \sum_{x \in T}
\frac{1}{\lambda} |f(x)|,
\end{multline*}
which is \eqref{treeweak}, as desired.  The proof of Theorem
\ref{free} is complete. \end{proof}


\section{Sharpness}

The purpose of this section is to prove Theorem \ref{dyadic-1} and Theorem \ref{micro-ex}.

\subsection{A preliminary construction}  Before we exhibit the full examples, we first
need a preliminary example of a maximal operator associated to a
finite Abelian group (but not to a metric) which has bad weak
$(1,1)$ behavior.

\begin{proposition}[Preliminary example]\label{prelim}  Let $q$ be a power of an odd prime,
and let $\FF_q$ be the finite field with $q$ elements.  If $q$ is
sufficiently large then there exists a vector space $X_q$ over
$\FF_q$ with counting measure $\mu$ and dimension $m = \dim(X_q)
\leq \sqrt q$, and disjoint sets $\{E_z\subseteq X_q\}_{z\in \FF_q}$ which
are symmetric around the origin (i.e. $x \in E_z$ if and only if $-x
\in E_z$) with measure
\begin{equation}\label{mujq}
 z\in F_q\implies \frac{1}{2q} \mu(X_q) < \mu(E_z) < \frac{2}{q} \mu(X_q),
\end{equation}
and such that the maximal function
$$ M_q f(x) \stackrel{\mathrm{def}}{=} \max_{z\in \FF_q} \frac{1}{\mu(E_z)} \int_{E_z} |f(x+y)| d\mu(y)$$
obeys the bounds
\begin{equation}\label{mefp}
\| M_q f \|_{L_p(X_q)} \lesssim
\left(\frac{p}{p-1}\right)^2\|f\|_{L_p(X_q)}
\end{equation}
for all $1 < p \leq \infty$,  but such that
\begin{equation}\label{mee}
\| M_q \|_{L_1(X_q) \to L_{1,\infty}(X_q)} > \frac{q}{2}.
\end{equation}
Furthermore, there exists a one-dimensional subspace $W_{-m+1}$ in
$X_q$ with the property that for all $z\in \FF_q$
\begin{equation}\label{vne}
 \mu( W_{-m+1} + E_z ) \geq \frac{1}{4} \mu(X_q),
\end{equation}
where $W_{-m+1}+E_z$ is the Minkowski sum of $W_{-m+1}$ and $E_z$.
\end{proposition}

\begin{remark} {\em The dimension bound $\dim(X_q) \le \sqrt q$ is not necessary for
Theorem \ref{dyadic-1}, but will be useful for proving Theorem
\ref{micro-ex}. Conversely, the property \eqref{vne} is used for
Theorem \ref{dyadic-1} but not for Theorem \ref{micro-ex}. Even though our choice of notation for $W_{-m+1}$ seems somewhat cumbersome at this juncture, it will become convenient when we apply Proposition~\ref{prelim} in Section~\ref{countersec}. }
\end{remark}

\begin{proof}[Proof of Proposition~\ref{prelim}]  Let $m$ be the largest integer less than $\sqrt q$. We set $X_q\stackrel{\mathrm{def}}{=} \FF_q^m$ to be the
 $m$-dimensional vector space
over $\FF_q$, with counting measure $\mu$.  On this space we
consider the non-degenerate quadratic form\footnote{One could also
use here a random symmetric function from $F_q^m$ to $F_q$  if
desired; the key features of $Q$ that we shall need are that it is
even, and its Fourier coefficients are all small.} $Q: X_q \to
\FF_q$ by
$$ Q(x_1,\ldots,x_m) \stackrel{\mathrm{def}}{=} x_1^2 + \ldots + x_m^2.$$
Define for $z\in \mathbb \FF_q$
$$ E_z \stackrel{\mathrm{def}}{=} \{ x \in \FF_q^m: Q(x) = z \}=Q^{-1}(z).$$
  Clearly  $E_z$ is
symmetric around the origin.

Let $\FF_q^*$ denote the dual of the additive group of $\FF_q$. Fix
a non-trivial character $\chi\in \FF_q^*\setminus \{1\}$. Then a
standard Gauss sum argument (see Lemma 4.14 in~\cite{TV06}) shows
that since $q$ is odd,
\begin{eqnarray}\label{eq:gauss}\left|\sum_{x\in \FF_q}
\chi\left(yx^2\right)\right|=\sqrt{q}\end{eqnarray} for every $y\in
\FF_q\setminus\{0\}$.

For every $x=(x_1,\ldots,x_m),x'=(x'_1,\ldots,x'_m)\in X_q$ write
$\langle x,x'\rangle \stackrel{\mathrm{def}}{=} \sum_{j=1}^m x_jx'_j\in \FF_q$. Then
for every $\eta\in X_q$ and $y\in\FF_q\setminus\{0\}$ we have (using
the fact that $q$ is odd),
\begin{multline}\label{eq:gauss2}
\left|\int_{X_q}\chi\left(yQ(x)+\langle
\eta,x\rangle\right)d\mu(x)\right| =\prod_{j=1}^m \left|\sum_{x_j\in
\FF_q} \chi\left(yx_j^2+\eta_jx_j\right)\right|\\=\prod_{j=1}^m
\left|\sum_{x_j\in \FF_q}
\chi\left(y\left(x_j+\frac{\eta_j}{2y}\right)^2\right)\right|\stackrel{\eqref{eq:gauss}}{=}q^{m/2}=\frac{\mu(X_q)}{q^{m/2}}.
\end{multline}

Consider the elementary identity
\begin{eqnarray}\label{eq:id}
\1_{E_z}(x)=\frac{1}{q} \sum_{y\in \FF_q}
\chi(-yz)\chi\left(yQ(x)\right).
\end{eqnarray}
For every $\eta \in X_q$  and $z\in \FF_q$ write
$$
\widehat \1_{E_z}(\eta)\stackrel{\mathrm{def}}{=} \frac{1}{\mu(X_q)} \int_{X_q}
\1_{E_z}(x)\chi\left(\langle \eta,x\rangle\right)d\mu(x).
$$
Then
\begin{multline}\label{eq:measure Ez}
\left|\frac{\mu\left(E_z\right)}{\mu(X_q)}-\frac{1}{q}\right|=\left|\widehat
\1_{E_z}(0)-\frac{1}{q}\right|\stackrel{\eqref{eq:id}}{\le}
\frac{1}{q\mu(X_q)}\sum_{y\in \FF_q\setminus\{0\}} \left|\int_{X_q}
\chi\left(yQ(x)+\langle
\eta,x\rangle-yz\right)d\mu(x)\right|\\\stackrel{\eqref{eq:gauss2}}{\le}
\frac{1}{q^{m/2}}\le \frac{\sqrt{q}}{q^{\frac12\sqrt{q}}}\le
\frac{1}{2q},
\end{multline}
provided that $q$ is large enough. This proves~\eqref{mujq}.
Moreover, for every $\eta\in X_q\setminus \{0\}$,
\begin{eqnarray}\label{eq:eta}
\left|\widehat \1_{E_z}(\eta)\right|\stackrel{\eqref{eq:id}}{\le}
\frac{1}{q\mu(X_q)}\sum_{y\in \FF_q\setminus
\{0\}}\left|\int_{X_q}\chi \left(yQ(x)+\langle\eta,x\rangle\right)
d\mu(x) \right|\stackrel{\eqref{eq:gauss2}}{\le} \frac{1}{q^{m/2}}\le
\frac{\sqrt{q}}{q^{\frac12\sqrt{q}}}.
\end{eqnarray}

Consider the averaging operator
$$
A_zf(x)\stackrel{\mathrm{def}}{=} \frac{1}{\mu(E_z)} \int_{E_z} f(x+y)d\mu(y).
$$
Inequalities~\eqref{eq:measure Ez} and~\eqref{eq:eta}, combined with
Parseval's identity, imply the $L_2$ bound
\begin{eqnarray}\label{eq:parseval}
\left\|f-A_zf\right\|_{L_2(X_q)}\le \frac{\mu(X_q)}{\mu(E_z)}\cdot
\max_{\eta\in X_q\setminus \{0\}} \left|\widehat
\1_{E_z}(\eta)\right|\cdot \|f\|_{L_2(X_q)}\le
\frac{2\|f\|_{L_2(X_q)}}{q^{\frac{m}{2}-1}}.
\end{eqnarray}
On the other hand, since $A_z$ is a contraction in $L_1$, we have
\begin{eqnarray}\label{eq:contraction}
\left\|f-A_zf\right\|_{L_1(X_q)}\le 2\|f\|_{L_1(X_q)}.
\end{eqnarray}
Interpolating between~\eqref{eq:parseval} and~\eqref{eq:contraction}
(see~\cite{Zyg02}) we get that for every $1\le p \le 2$,
\begin{eqnarray}\label{eq:interpol}
\left\|f-A_zf\right\|_{L_p(X_q)}\le 2^{\frac{2}{p}-1}\cdot
\left(\frac{2}{q^{\frac{m}{2}-1}}\right)^{2-\frac{2}{p}}\|f\|_{L_p(X_q)}=
2\left(q^{1-\frac{m}{2}}\right)^{2-\frac{2}{p}}\|f\|_{L_p(X_q)}.
\end{eqnarray}
Hence
\begin{multline*}
\left\|\max_{z\in \FF_q} \big||f|-A_z\left(|f|\right)\big|\right\|_{L_p(X_q)}\le
\left(\sum_{z\in \FF_q}\Big\|
\big||f|-A_z\left(|f|\right)\big|\Big\|_{L_p(X_q)}^p\right)^{1/p}\\\stackrel{\eqref{eq:interpol}}{\le}
2q^{1/p}\left(q^{1-\frac{m}{2}}\right)^{2-\frac{2}{p}}\|f\|_{L_p(X_q)}.
\end{multline*}
Thus
\begin{multline}\label{eq:getting the square}
\|M_qf\|_{L_p(X_q)}
\le\left(1+2q^{1/p}\left(q^{1-\frac{m}{2}}\right)^{2-\frac{2}{p}}\right)\|f\|_{L_p(X_q)}\\\le\left(1+2q^{1/p}
\left(q^{\frac{3}{2}-\frac{\sqrt
q}{2}}\right)^{2-\frac{2}{p}}\right)\|f\|_{L_p(X_q)}\lesssim \frac{\|f\|_{L_p(X_q)}}{(p-1)^2}.
\end{multline}
The last step in~\eqref{eq:getting the square} can be proved as follows: for $q\ge 36$, the term $1+2q^{1/p}
\left(q^{\frac{3}{2}-\frac{\sqrt
q}{2}}\right)^{2-\frac{2}{p}}$ is $\lesssim 1+ q^{1-\e-\frac{\e\sqrt{q}}{2}}$, where we write $\frac{1}{p}=1-\e$. Now consider the cases $\e\ge \frac{2}{\sqrt{q}}$ and $\e<\frac{2}{\sqrt{q}}$ separately. The bound~\eqref{eq:getting the square} proves~\eqref{mefp} when $1\le p\le 2$. The case $p>2$ follows
from a similar interpolation argument, using trivial bound
$\|M_qf\|_{L_\infty(X_q)}\le \|f\|_{L_\infty(X_q)}$.

To prove~\eqref{mee}, let $f \stackrel{\mathrm{def}}{=} \1_{\{0\}}$ be the indicator
function of the origin $0$.  Then $\|f\|_{L_1(X)} = 1$.  Since the
sets $\{E_z\}_{z\in \FF_q}$ cover $X_q$, we see from~\eqref{mujq}
that $M_q f(x) > \frac{q}{2\mu(X_q)}$ for all $x \in X_q$,
and~\eqref{mee} follows (setting $\lambda$ slightly larger than
$\frac{q}{2\mu(X_q)}$).

Finally, let $W_{-m+1}$ be the span of the first basis vector
$e_1\in X_q=\FF_q^m$. Let $S$ denote the set of squares in $\FF_q$,
i.e. $S\stackrel{\mathrm{def}}{=} \{x^2:x\in \FF_q\}$. Since $q$ is odd,
$|S|=\frac{q+1}{2}$. Observe that $(x_1,\ldots,x_m)$ lies in $E_z +
W_{-m+1}$ if and only if $x_2^2 + \ldots + x_m^2$ is in $z-S$.
Arguing as in~\eqref{eq:measure Ez} we deduce that
\begin{eqnarray*}
\mu\left(E_z + W_{-m+1}\right)&=&\sum_{s\in
S}\mu\left((x_1,x_2,\ldots,x_m)\in \FF_q^{m}:\
x_2^2+\cdots+x_m^2=z-s\right)\\&=&\sum_{s\in
S}q\left|\{(x_2,\ldots,x_m)\in \FF_q^{m-1}:\
x_2^2+\cdots+x_m^2=z-s\}\right|\\&\ge&\sum_{s\in S}
q\left(\frac{|\FF_q^{m-1}|}{q}-\frac{|\FF_q^{m-1}|}{q^{(m-1)/2}}\right)
\\&=&\mu(X_q)\left(\frac{q+1}{2q}-\frac{q+1}{2q^{(m-1)/2}}\right).
\end{eqnarray*}
This establishes the bound~\eqref{vne} for $q$ sufficiently large.
\end{proof}

\begin{remark} {\em This example once again demonstrates the
(well-known) fact that $L_2$-type smoothing estimates, such as those
arising from smallness of Fourier coefficients, can imply $L_p$ maximal
 bounds by standard interpolation arguments, but do not necessarily imply weak-type $(1,1)$
 bounds.}
\end{remark}

\subsection{The doubling example}\label{countersec}

We now prove Theorem~\ref{dyadic-1}.  The claim is trivial for $K
\leq 48$, so we will assume $K \geq 48$.  By Bertrand's postulate we
may find an odd prime $q$ between $K/4$ and $K/2$, which we now fix.
We then let $\FF_q$, $X_q$ and $\{E_z\}_{z\in \FF_q}$ be as in
Proposition~\ref{prelim}. Fix an arbitrary enumeration of the points
in $\FF_q$, say $\FF_q=\{z_1,\ldots,z_q\}$ and write $E_{z_j}=E_j$
(this will not create any ambiguity in what follows). It will also be convenient to set $E_0=\{0\}$. The maximal
function $M_q$ in Proposition~\ref{prelim} is not associated to a
metric, let alone one with the doubling property~\eqref{eq:def doubling}, since
the sets $E_j$ are not nested. However, this can be remedied by
extending the space $X_q$ in the following fashion.

We let $X \stackrel{\mathrm{def}}{=} X_q \times \FF_q^q$ be the Cartesian product of
$X_q$ with the vector space $\FF_q^q$, with counting measure $\mu$.
We also let
$$ \{0\} = V_0 \subseteq V_1 \subseteq \ldots \subseteq V_q = \FF_q^q$$
be the standard flag in $\FF_q^q$, thus $V_j$ is the span of $\{e_1,\ldots,e_j\}$ for $j\in \{0,\ldots,q\}$, where $e_1,\ldots,e_q$ is the standard basis of $\FF_q^q$.  In particular
\begin{equation}\label{vqj}
j\in\{0,\ldots,q\}\implies \mu(V_j) = q^j.
\end{equation}
Recall that $X_q$ is itself a vector space $\FF_q^m$ over $\FF_q$,
thus we have another flag
$$ \{0\} = W_{-m} \subseteq W_{-m+1} \subseteq \ldots \subseteq W_0 = X_q,$$
where
\begin{equation}\label{wqj}
j\in \{0,\ldots,m\}\implies \mu(W_{-j}) = \frac{\mu(X_q)}{q^j}.
\end{equation}
We can ensure that $W_{-m+1}$ is the one-dimensional subspace mentioned in Proposition \ref{prelim}.

For $v\in \FF_q$ let $j(v)$ denote the minimal $j\in \{0,\ldots,q\}$ such that $v\in V_j$. For $u\in X_q$ and $j\in \{0,\ldots,q\}$, let $\ell_j(u)$ be the maximal $\ell\in \{0,\ldots,m\}$ such that $u\in E_j+W_{-\ell}$. Now, for $(u,v),(u',v')\in X$ define
\begin{equation}\label{eq:def d}
d((u,v),(u',v')) \stackrel{\mathrm{def}}{=} 4^{j(v-v')}\1_{\{v\neq v'\}}+2^{-\ell_{j(v-v')}(u-u')}\1_{\{u\neq u'\}}.
\end{equation}

We claim that $d$ is a translation invariant metric on $X$. The translation invariance and non-degeneracy of $d$ are immediate from the definition. The symmetry of $d$ follows from the fact that the $E_j\subseteq X_q$ are symmetric around the origin. It therefore remains to verify that for all $x,y\in X$ we have $d(x+y,0)\le d(x,0)+d(y,0)$. Write $x=(u,v)$, $y=(u',v')$, $j=j(v)$, $j'=j(v')$, $\ell=\ell_{j}(u)$, $\ell'=\ell_{j'}(u')$. Without loss of generality $j\ge j'$. Then $v\in V_j$ and $v'\in V_{j'}\subseteq V_j$. So, $v+v'\in V_j$, i.e., $j(v+v')\le j$.  Denoting $\ell''=\ell_{j(v+v')}(u+u')$, we see that it suffices to prove the inequality
\begin{equation}\label{eq:goal sub addtivie}
4^{j}\1_{\{v+v'\neq 0\}}+2^{-\ell''}\1_{\{u+u'\neq 0\}}\le 4^j\1_{\{v\neq 0\}}+2^{-\ell}\1_{\{u\neq 0\}}+4^{j'}\1_{\{v'\neq 0\}}+2^{-\ell'}\1_{\{u'\neq 0\}}
\end{equation}
If $j'\ge 1$ then $v,v'\neq 0$, and~\eqref{eq:goal sub addtivie} holds since $4^{j'}\ge 4\ge 2^{-\ell''}$. On the other hand, if $j'=0$ (equivalently $v'=0$) then by definition $u'\in W_{-\ell'}$. Since $u\in E_j+W_{-\ell}$, it follows that $u+u'\in E_j+W_{-\ell}+W_{-\ell'}=E_{j(v+v')}+W_{-\min\{\ell,\ell'\}}$. Thus $\ell''\ge \min\{\ell,\ell'\}$, and~\eqref{eq:goal sub addtivie} follows from the trivial inequality  $2^{-\min\{\ell,\ell'\}}\1_{\{u+u'\neq 0\}}\le 2^{-\ell}\1_{\{u\neq 0\}}+ 2^{-\ell'}\1_{\{u'\neq 0\}}$.


The balls in the metric $d$ take the following form:
\begin{equation}\label{eq:range1}
r\ge 4^q+1\implies B(0,r)=X,
\end{equation}
\begin{equation}\label{eq:range2}
\exists j\in \{1,\ldots,q-1\},\  4^j+1\le r<4^{j+1}\implies B(0,r)=X_q\times V_j,
\end{equation}
\begin{equation}\label{eq:range3}
\exists j\in \{1,\ldots,q-1\},\  4^j \le r< 4^j+2^{-m+1}\implies B(0,r)=\left(E_j\times V_j\right)\cup \left(X_q\times
V_{j-1}\right),
\end{equation}
\begin{multline}\label{eq:range4}
\exists (j,\ell)\in \{1,\ldots,q-1\}\times \{1,\ldots, m-1\},\ 4^j+2^{-\ell}\le r< 4^j+2^{-\ell+1}\\\implies B(0,r)=\left((E_j+W_{-\ell})\times V_j\right)\cup \left(X_q\times
V_{j-1}\right),
\end{multline}
\begin{equation}\label{eq:range5}
1\le r<4\implies B(0,r)=X_q\times \{0\},
\end{equation}
\begin{equation}\label{eq:range6}
\exists \ell\in \{1,\ldots, m\},\ 2^{-\ell}\le r<2^{-\ell+1}\implies B(0,r)=W_{-j}\times \{0\}.
\end{equation}

We shall first of all prove that $(X,d,\mu)$ is doubling with constant
$2q\le K$. For $r\ge 4$ take $j\in \{1,2,\ldots\}$ such that $4^j\le r
<4^{j+1}$. If, in addition, $4^j+1\le r<4^{j+1}$ then since $2r<4^{j+2}$, it follows from~\eqref{eq:range2}, \eqref{eq:range3}, \eqref{eq:range4}  that $B(x,2r)\subseteq X_q\times V_{j+1}$, implying that
\begin{equation}\label{eq:first doubling}
\mu\left(B(0,2r)\right)\le \mu\left(X_q\times
V_{j+1}\right)=q^{j+1}\mu\left(X_q\right)\le
q\cdot\mu\left(X_q\times V_j\right)\stackrel{\eqref{eq:range2}}{=}q\cdot\mu\left(B(0,r)\right).
\end{equation}
On the other hand, if $4^j\le r<4^j+1$ then $4^j+1\le 2r< 4^{j+1}$. Note that~\eqref{eq:range3}, \eqref{eq:range4} imply that
\begin{equation}\label{eq:trivial inclusion}
E_j\times V_j\subseteq B\left(0,r\right),
\end{equation}
and therefore
\begin{equation*}
\mu\left(B(0,2r)\right)\stackrel{\eqref{eq:range2}}{=}\mu\left(X_q\times V_j\right)=q^j\mu(X_q)\stackrel{\eqref{mujq}}{\le}
2q^{j+1}\mu\left(E_j\right)=2q\mu\left(E_j\times V_j\right)\stackrel{\eqref{eq:trivial inclusion}}{\le}2q\mu\left(B\left(x,r\right)\right).
\end{equation*}
Similarly, using~\eqref{eq:range5}, \eqref{eq:range6}, also for $0<r<4$ we have $\mu\left(B(0,2r)\right)\le q\mu\left(B(0,r)\right)$. Thus $(X,d,\mu)$ is doubling with constant $2q$, as claimed.

Now, from \eqref{mee} we can find $f_q: X_q \to \R_+$ with norm $\|f_q\|_{L_1(X_q)} = 1$ and $\lambda > 0$ such that
\begin{equation}\label{eq:bad f} \mu( M_q f_q > \lambda ) > \frac{q}{2\lambda}.\end{equation}
We extend this function $f_q$ to a function $f: X \to \R_+$ defined
by $f(x,y) \stackrel{\mathrm{def}}{=} f_q(x)$ for $x \in X_q$ and $y \in \FF_q^q$.
Thus
\begin{equation}\label{eq:norm bad f} \|f\|_{L_1(X)} = |\FF_q^q| \cdot \|f_q\|_{L_1(X_q)} = q^q.\end{equation}

We shall next compute $M_{2^\Z} f(x,y)$ for $(x,y) \in X_q \times \FF_q^q = X$.  Actually, for very minor
technical reasons we need to consider the slight variant
\begin{equation}\label{eq:def Meps} M^\eps_{2^\Z} f(x,y) = \sup_{r \in 2^\Z} \frac{1}{\mu(B((x,y),(1+\eps) r))} \int_{B((x,y),(1+\eps)r)} |f(x,y)|\ d\mu(x,y)\end{equation}
for some small $\eps > 0$,
but this clearly will not make a difference since we can rescale the metric by $1+\eps$.

Observe that for any $1 \leq j \leq q$ we have
$$ M^\eps_{2^\Z} f(x,y) \geq \frac{1}{\mu\left(B(0,(1+\e)4^j)\right)} \sum_{(x',y') \in B(0,(1+\e)4^j)} f_q(x+x').$$
Note that if $0<\e<4^{-q}\cdot 2^{-m+1}$ then it follows from~\eqref{eq:range3} that
\begin{equation*}
\mu\left(B(0,(1+\e)4^j)\right)\le \mu(E_j\times V_j)+\mu(X_q\times V_{j-1})\stackrel{\eqref{mujq}}{\le} \frac{2}{q}\mu(X_q)q^j+\mu(X_q)q^{j-1}=3q^{j-1}\mu(X_q).
\end{equation*}
Using the inclusion $B(0,(1+\e)4^j)\supseteq E_j\times V_j$, which trivially follows from~\eqref{eq:range3}, we conclude that
$$ M^\eps_{2^\Z} f(x,y) \geq \frac{1}{3 q^{j-1} \mu(X_q)} \sum_{x' \in E_j} \sum_{y' \in V_j} f_q(x+x').$$
Hence, in combination with \eqref{vqj} and \eqref{mujq}, we get the bound
$$ M^\eps_{2^\Z} f(x,y) \geq \frac{1}{6 \mu(E_j)} \sum_{x' \in E_j} f_q(x+x').$$
Taking the supremum over all $j$ we conclude the pointwise estimate
$$ M^\eps_{2^\Z} f(x,y) \geq \frac{1}{6} M_q f_q(x).$$
In particular we have
$$ \mu\left( M^\eps_{2^\Z} f > \frac{1}{6} \lambda  \right) \geq |\FF_q^q| \mu\left(  M_q f_q > \lambda  \right)
> q^q \cdot \frac{q}{2\lambda}.$$
Recalling~\eqref{eq:norm bad f} we thus see that
$$ \| M^\eps_{2^\Z} \|_{L_1(X) \to L_{1,\infty}(X)} \geq \frac{1}{12} q \geq \frac{K}{48},$$
yielding \eqref{mosh}.

The only remaining task is to establish the $L_p$ bounds $\|M\|_{L_p(X)\to L_p(X)}\lesssim_p 1$, for $p>1$. To do this let's examine what equations~\eqref{eq:range1}--\eqref{eq:range6} say about the measures of the balls $B(0,r)$ appearing in the definition of the maximal function $M$. For $r < 4$, the balls all take the form $W_{-j} \times \{0\}$ for some $-m \leq -j \leq 0$. For $4^j \le r < 4^j + 2^{-m+1}$ for some $1 \leq j \leq q$, the ball $B(0,r)$ is equal to the union of the two sets
$E_j \times V_j$ and $X_q \times V_{j-1}$, which have the same measure up to a universal factor thanks to \eqref{mujq}, \eqref{vqj}.  For $4^j+2^{-m+1} \le r < 4^{j+1}$, we see that the ball $B(0,r)$ lies between $(E_j+W_{-m+1}) \times V_j$ and $X_q \times V_j$, and so thanks to~\eqref{vne} has measure comparable to $X_q \times V_q$.  Putting all this together, we obtain the pointwise bound

\begin{multline}\label{mix}
M g(x,y) \lesssim  \max_{-m \leq -j \leq 0} \frac{1}{\mu(W_{-j})} \sum_{x' \in x + W_{-j}} |g(x',y)|
 + \max_{0 \leq j \leq q} \frac{1}{\mu(X_q) \mu(V_j)} \sum_{x' \in X_q} \sum_{y' \in y + V_j} |g(x',y')|\\
+ \max_{0 \leq j \leq q} \frac{1}{\mu(E_j) \mu(V_j)} \sum_{x' \in x + E_j} \sum_{y' \in y + V_j} |g(x',y')|
\end{multline}
for all functions $g:X\to \R$.

If we let $\mathscr{B}_{-j}$, for $-m \leq -j \leq 0$, be the $\sigma$-algebra on $X$ generated by the cosets of $W_{-j} \times \{0\}$, we have
\begin{equation}\label{eq:use doob Bj} \max_{-m \leq -j \leq 0} \frac{1}{\mu(W_{-j})} \sum_{x' \in x + W_{-j}} |g(x',y)|
= \max_{-m \leq -j \leq 0} \E\left[ |g| \big | \B_{-j} \right](x,y),\end{equation}
where $\E\left[ |g| \big | \B_{-j} \right]$ denotes the conditional expectation of $|g|$ with respect to the $\sigma$-algebra $\B_{-j}$.  Applying Doob's maximal inequality (Proposition~\ref{doob}), we thus see that this expression is bounded on $L_p$, i.e.,
\begin{equation}\label{eq:use doob first term}
\left(\int_X\left|\sup_{-m \leq -j \leq 0} \frac{1}{\mu(W_{-j})} \sum_{x' \in x + W_{-j}} |g(x',y)|\right|^pd\mu(x,y)\right)^{1/p}\le \frac{p}{p-1}\|g\|_{L_p(X)}.
\end{equation}
A similar argument disposes of the second term in \eqref{mix}, i.e.,
  \begin{equation}\label{eq:use doob second term}
\left(\int_X\left|\max_{0 \leq j \leq q} \frac{1}{\mu(X_q) \mu(V_j)} \sum_{x' \in X_q} \sum_{y' \in y + V_j} |g(x',y')|\right|^pd\mu(x,y)\right)^{1/p}\le \frac{p}{p-1}\|g\|_{L_p(X)}.
\end{equation}
By combining~\eqref{eq:use doob first term} and~\eqref{eq:use doob second term} with~\eqref{mix}, we see that it suffices to establish the bound
  \begin{equation}\label{eq:goal third term}
\left(\int_X\left|\max_{0 \leq j \leq q} \frac{1}{\mu(E_j) \mu(V_j)} \sum_{x' \in x + E_j} \sum_{y' \in y + V_j} |g(x',y')|\right|^pd\mu(x,y)\right)^{1/p}\lesssim \left(\frac{p}{p-1}\right)^3\|g\|_{L_p(X)}.
\end{equation}
We can bound the left-hand side of~\eqref{eq:goal third term} by
 \begin{equation*}\label{eq:G appears}
\left(\int_X\left|\max_{0 \leq j \leq q} \frac{1}{\mu(V_j)} \sum_{y' \in y + V_j} G(x,y')\right|^pd\mu(x,y)\right)^{1/p},
\end{equation*}
where
$$ G(x,y') \stackrel{\mathrm{def}}{=} \max_{0 \leq j \leq q} \frac{1}{\mu(E_j)} \sum_{x' \in x + E_j} |g(x',y')|.$$
Applying Doob's maximal inequality again, we thus reduce to showing that
$$ \| G \|_{L_p(X)} \lesssim \left(\frac{p}{p-1}\right)^2\|g\|_{L_p(X)}.$$
But this follows from \eqref{mefp} (and Fubini's theorem).  The proof of Theorem \ref{dyadic-1} is complete.\qed

\subsection{The Ahlfors-David regular example}\label{sec:AD}

Now we prove Theorem \ref{micro-ex}.  Once again we may take $n$ to
be large, as the claim is easy for bounded $n$ (e.g., one could take
the usual Hardy-Littlewood maximal function on $\R^n$).

The heart of our construction is the following lemma:

\begin{lemma}\label{lem:finite AD}
There exists a finite Abelian group $X$, equipped with counting measure $\mu$ and an invariant metric $d_X$, with the following properties:
\begin{enumerate}
\item The are integers $a<b$ such that for all $x,y\in X$ we have $d_X(x,y)\in \{0\}\cup\{3^{j/n}\}_{j=a}^b$.
\item For all $r\in [3^{a/n},3^{b/n}]$ and all $x\in X$ we have \begin{equation}\label{eq:AD finite}
    3^{-a} r^n\le \mu(B(x,r))\le 3^{-a+3}r^n.\end{equation}
    \item $\|M\|_{L_1(X)\to L_{1,\infty}(X)}\gtrsim n\log n$.
    \item $\left\|M_{2^\Z}\right\|_{L_1(X)\to L_{1,\infty}(X)}\gtrsim \log n$.
    \item For all $1<p\le \infty$ we have $\|M\|_{L_p(X)\to L_p(X)}\lesssim_p 1$.
\end{enumerate}
\end{lemma}

\begin{proof}[Proof of Theorem~\ref{micro-ex} assuming Lemma~\ref{lem:finite AD}]
 In what follows $\FF_3$ denotes the field of size $3$. Let $Y$ be the subspace of $\FF_3^{\aleph_0}$ consisting of all finitely supported vectors, equipped with the counting measure $\nu$. For $(y_1,y_2,\ldots)\in Y$ let $j(y)$ denote the largest  $j\in \N$ such that $y_j\neq 0$. If $y=0$ we set $j(y)=-\infty$. For $y,y'\in Y$ define
$$
\rho_Y(y,y')\stackrel{\mathrm{def}}{=} 3^{b/n}\cdot3^{j(y-y')/n}.
$$
Then $\rho_Y$ is an invariant ultrametric on $Y$, satisfying $\rho_Y(y,y')\in \{0\}\cup\{3^{(b+j)/n}\}_{j=1}^\infty$ for all $y,y'\in Y$. Let $Y_j\subseteq Y$ denote the set of vectors whose support is contained in the first $j$ coordinates. Thus $Y_j$ is a subspace of $Y$ and $Y_j=B_{\rho_Y}\left(0,3^{(b+j)/n}\right)$. Since $\nu(Y_j)=3^j$, it follows that for all $r\ge 3^{b/n}$ and $y\in Y$ we have
\begin{equation}\label{eq:nu AD}
3^{-b-1} r^n\le \nu\left(B_{\rho_Y}(y,r)\right)\le 3^{-b}r^n.
\end{equation}

Next, we let $Z$ denote the set $\FF_3^{\aleph_0}$, and let $\tau$ denote the countable product of the normalized counting measure on $\FF_3$. Thus $\tau$ is an invariant probability measure on $Z$. For $k\in \N$ let $Z_k$ be the subspace of $Z$ consisting of $(z_1,z_2,\ldots)\in Z$ with $z_1=z_2=\ldots=z_k=0$ (we shall also use the convention $Z_0=Z$). Thus $\tau(Z_k)=3^{-k}$. For $z\in Z$ let $k(z)$ denote the largest integer $k\ge 0$ such that $z\in Z_k$ (with the convention $k(0)=\infty)$. For $z,z'\in Z$ define
$$
\rho_Z(z,z') \stackrel{\mathrm{def}}{=} 3^{(a-1)/n}\cdot 3^{-k(z-z')/n}.
$$
Then $\rho_Z$ is an invariant ultrametric on $Z$, satisfying $\rho_Z(z,z')\in \{0\}\cup\{3^{(a-j)/n}\}_{j=1}^\infty$ for all $z,z'\in Y$. It follows from the definitions that for all $k\ge 0$ we have  $B_{\rho_Z}\left(0,3^{(a-k-1)/n}\right)=Z_k$. Let $\sigma$ be the invariant measure on $Z$ given by $\sigma=3^{a-1}\tau$. Thus for all $r\le 3^{a/n}$ we have,
\begin{equation}\label{eq:sigma AD}
\frac13 r^n\le \sigma\left(B_{\rho_Z}(y,r)\right)\le 3r^n.
\end{equation}

We shall now let $G$ be the Abelian group $Z\times X\times Y$, equipped with the $\ell_\infty$ product metric
$
d_G\left((z,x,y),(z',x',y')\right)=\max\left\{\rho_Z(z,z'),d(x,x'),\rho_Y(y,y')\right\}.
$
We shall also equip $G$ with the product measure $\mu_G=\sigma\times \mu\times \nu$.

The balls in $G$ are given by $B_{d_G}(0,r)=B_{\rho_Z}(0,r)\times B_{d_X}(0,r)\times B_{\rho_Y}(0,r)$. If $r\ge 3^{b/n}$ then $B_{d_X}(0,r)=X$, and thus by~\eqref{eq:AD finite} we have
$\mu\left(B_{d_X}(0,r)\right)\in \left[3^{b-a},3^{b-a+3}\right]$. Similarly, for $r\ge 3^{b/n}$ we have $B_{\rho_Z}(0,r)=Z$, and thus $\sigma\left(B_{\rho_Z}(0,r)\right)=3^{a-1}$. It therefore follows from~\eqref{eq:nu AD} that
\begin{equation}\label{eq:AD big r}
r\ge 3^{b/n}\implies \mu_G\left(B_{d_G}(0,r)\right)\in \left[\frac{1}{9}r^n,9r^n\right].
\end{equation}
If $3^{a/n}\le r< 3^{b/n}$, then $B_{\rho_Y}(0,r)=\{0\}$, and hence $\nu\left(B_{\rho_Y}(0,r)\right)=1$. As before, we also have in this case $\sigma\left(B_{\rho_Z}(0,r)\right)=3^{a-1}$, and by~\eqref{eq:AD finite}, $\mu\left(B_{d_X}(0,r)\right)\in \left[3^{-a}r^n,3^{-a+3}r^n\right]$. Thus,
\begin{equation}\label{eq:AD medium r}
3^{a/n}\le r< 3^{b/n}\implies \mu_G\left(B_{d_G}(0,r)\right)\in \left[\frac{1}{3}r^n,9r^n\right].
\end{equation}
Finally, for $r<3^{a/n}$ we have $B_{\rho_Y}(0,r)=\{0\}$ and $B_{d_X}(0,r)=\{0\}$ and so $\nu\left(B_{\rho_Y}(0,r)\right)=\mu\left(B_{d_X}(0,r)\right)=1$. In combination with~\eqref{eq:sigma AD}, we see that
\begin{equation}\label{eq:AD small r}
r< 3^{a/n}\implies \mu_G\left(B_{d_G}(0,r)\right)\in \left[\frac{1}{3}r^n, 3r^n\right].
\end{equation}
Inequalities~\eqref{eq:AD big r}, \eqref{eq:AD medium r}, \eqref{eq:AD small r} show that the metric measure space $(G,d_G,\mu_G)$ is Ahlfors-David $n$-regular.

 It remains to prove the estimates \eqref{micro-1}, \eqref{micro-2}, \eqref{micro-3}. By assertion $(3)$ of Lemma~\ref{lem:finite AD} we can find $f: X \to \R_+$ with $\|f\|_{L_1(X)} = 1$, and $\lambda > 0$, such that
\begin{equation}\label{eq:assumption lower f}
\mu\left( M f > \lambda \right) \gtrsim \frac{n\log n}{\lambda}.
\end{equation}
Define a function $g:G\to \R_+$ by $g(z,x,y)=f(x)\1_{\{y=0\}}$. Then $\|g\|_{L_1(G)}=\sigma(Z)=3^{a-1}$. Moreover, we have the pointwise estimate
\begin{multline*}
Mg(x,y,z)\ge \sup_{3^{a/n}\le r\le 3^{b/n}}\frac{\int_{B_{\rho_Z}(0,r)\times B_{d_X}(0,r)\times B_{\rho_Y}(0,r)}f(x+x')\1_{\{z+z'=0\}}d\mu_G(z',x',y')}{\mu_G\left(B_{\rho_Z}(0,r)\times B_{d_X}(0,r)\times B_{\rho_Y}(0,r)\right)}\\=\left(Mf(x)\right)\1_{\{z=0\}},
\end{multline*}
where we used the fact that $B_{\rho_Y}(0,r)=\{0\}$ for $r\le 3^{b/n}$.
Thus by Fubini's theorem,
$$
\mu_G\left(Mg>\lambda\right)\ge \sigma(Z)\mu\left(Mf>\lambda\right)\stackrel{\eqref{eq:assumption lower f}}{\gtrsim} 3^{a-1}\frac{n\log n}{\lambda}=\frac{n\log n}{\lambda}\|g\|_{L_1(G)}.
$$
This proves~\eqref{micro-1}; the proof of~\eqref{micro-2} is identical.
To prove~\eqref{micro-3} take an non-negative $h\in L_p(G)$, and observe the pointwise bound
\begin{eqnarray}
Mh(z,x,y)&\le& \label{first term product} \sup_{r<3^{a/n}}\frac{\int_{B_{\rho_Z}(0,r)}h(z+z',x,y)d\sigma(z')}{\sigma\left(B_{\rho_Z}(0,r)\right)}
\\&\phantom{\le}&\label{second term product}+ \sup_{3^{a/n}\le r\le 3^{b/n}}\frac{\int_{Z\times B_{d_X}(0,r)}h(z',x+x',y)d\sigma(z')d\mu(x')}{2^{a-1}\mu\left(B_{d_X}(0,r)\right)}\\&\phantom{\le}&\label{third term product}+
\sup_{r> 3^{b/n}}\frac{\int_{Z\times X\times B_{\rho_Y}(0,r)}h(z',x',y+y')d\sigma(z')d\mu(x')d\nu(y')}{2^{a-1}\cdot3^{b-a}\nu\left(B_{\rho_Y}(0,r)\right)},
\end{eqnarray}
where in the denominator of~\eqref{third term product} we used the fact that $\mu(x)\ge 3^{b-a}$.

Since $\rho_Z$ is an ultrametric, Doob's maximal inequality implies that for all $x\in X$ and $y\in Y$ we have,
$$
\int_Z \left(\sup_{r<3^{a/n}}\frac{\int_{B_{\rho_Z}(0,r)}h(z+z',x,y)d\sigma(z')}{\sigma\left(B_{\rho_Z}(0,r)\right)}\right)^p
d\sigma(z)\lesssim_p \int_Z h(z,x,y)^pd\sigma(z).
$$
Thus, by Fubini's theorem, the $L_p(G)$ norm of the term in~\eqref{first term product} is $\lesssim_p \|h\|_{L_p(G)}$. A similar argument shows that the $L_p(G)$ norm of the term in~\eqref{third term product} is $\lesssim_p \|h\|_{L_p(G)}$. Finally, using assertion $(5)$ of Lemma~\ref{lem:finite AD}, we get the same bound for the term in~\eqref{second term product}, proving~\eqref{micro-3}.
\end{proof}

\begin{proof}[Proof of Lemma~\ref{lem:finite AD}]
Let $q=3^k$ be a power of three between $\frac{1}{3} n \log n$ and
$\frac{1}{9} n \log n$. We invoke Proposition \ref{prelim} to
create a vector space $X_q = \FF_q^m$ over a finite field $\FF_q$
with counting measure $\mu$, together with sets $E_1,\ldots,E_q$
obeying the properties stated in Proposition~\ref{prelim}; in
particular
\begin{equation}\label{nbound}
m \lesssim \sqrt{n \log n}.
\end{equation}
Note that $\FF_q$ can itself be viewed as a vector space over the
field $\FF_3$ of three elements, and thus $X_q$ is a vector space
over $\FF_3$ of dimension
\begin{equation}\label{Mdef}
M \stackrel{\mathrm{def}}{=} mk= m \log_3 q \lesssim n^{1/2} (\log n
)^{3/2}.
\end{equation}

As in Section~\ref{countersec}, the idea is to take a Cartesian product of $X_q$ with another vector space, and try to create balls which resemble the product of a set $E_j$ with a subspace.  Some care is however required in order to make the construction compatible with both the constraint \eqref{eq:def micro} and the triangle inequality.

Analogously to the arguments in Section~\ref{countersec}, we shall need a flag
$$ \{0\} = W_{-M} \subseteq W_{-M+1} \subseteq \ldots \subseteq W_0 = X_q$$
of vector spaces over $\FF_3$ in $X_q$, so that $\mu( W_{-j} ) =
3^{-j} \mu( X_q )$ for all $-M \leq -j \leq 0$.  (We will not use
\eqref{vne} or the space $W_{-m+1}$ in Proposition \ref{prelim}, so
there is no collision of notation here.)

Our space shall
be $X \stackrel{\mathrm{def}}{=} X_q \times \FF_3^{q}$, with
counting measure $\mu$. We shall need a flag
$$ \{0\}=V_0 \subseteq V_1 \subseteq \ldots \subseteq V_{q} = \FF_3^{q}$$
in $\FF_3^{q}$, with $\mu(V_j) = 3^j$.

For every integer $-M \leq j \leq q$, we define the set $B_j
\subseteq X =  X_q \times \FF_3^{q}$ as follows:
\begin{itemize}
\item If $-M \leq j \leq 0$, we set $B_j \stackrel{\mathrm{def}}{=} W_j \times \{0\}$.
\item If $1 \leq j \leq q$, we set
\begin{equation}\label{eq:def Bj}B_j \stackrel{\mathrm{def}}{=} \left(X_q \times V_{j}\right) \bigcup \left(\bigcup_{\ell=1}^{\min\{j+k,q\}}
(E_{\ell} \times V_{\ell})\right).\end{equation}
\end{itemize}
The $B_j$ are symmetric and nested, with
\begin{equation}\label{bnest}
 \{0\} = B_{-M} \subseteq B_{-M+1} \subseteq \ldots \subseteq B_{q} = X.
 \end{equation}

We define a function $d: X \times X \to \R_+$ by setting $d(x,x) =
0$ for all $x \in X$, and
\begin{equation}\label{ddef}
d(x,y) \stackrel{\mathrm{def}}{=} \min \left\{
3^{j/n}: x-y \in B_j \right\},
\end{equation}
for all distinct $x,y \in X$.  Thus $d$ takes values in $\{0\} \cup
\left\{ 3^{j/n}:\  -M+1 \leq j \leq q
\right\}$. The first assertion of Lemma~\ref{lem:finite AD} therefore holds with $a=-M+1$ and $b=q$.

\begin{claim} $d$ is a translation-invariant metric on $X$.
\end{claim}

\begin{proof} The translation-invariance, non-degeneracy,
 and symmetry properties of $d$ are obvious (symmetry follows from the symmetry of $E_j$). The only non-trivial task is to
 verify the triangle inequality.  By construction, it will suffice to show that
$ x + x' \in B_{j''-1}$
whenever $x \in B_j$, $x' \in B_{j'}$, and $-M < j, j', j'' \leq
q$ are such that
\begin{equation}\label{eq:triangle ineq assumption}
3^{j''/n}
> 3^{j/n} +
3^{j'/n}.
\end{equation}

By symmetry we may assume that $j \leq j'$. It follows from~\eqref{eq:triangle
ineq assumption} that provided $n$ is large enough,
\begin{equation}\label{eq:j+s}
3^{j''/n}> 3^{j'/n}\left(1+3^{-(M+q-1)/n}\right)\ge 3^{j'/n}\left(1+3^{-\frac12 \log n}\right)\ge 3^{(j'+k)/n},
\end{equation}
where we used the fact that $q\le \frac{1}{3}n\log n$, while $M\lesssim n^{1/2}(\log n)^{3/2}$ and $k =\log_3q\lesssim \log n$.

It follows from~\eqref{eq:j+s} that
\begin{equation}\label{eq:j+s2}
j''>j'+k.
\end{equation}
If $j' \le 0$, then we
have $B_j + B_{j'} = B_{j'}$, so $x+x' \in B_{j'} \subseteq
B_{j''-k}\subseteq B_{j''-1}$, as required. Assume therefore that $j' \geq 1$. Then
$B_j \subseteq B_{j'} \subseteq X_q \times V_{\min\{j'+k,q\}}$, and
hence $x+x' \in X_q \times V_{\min\{j'+k,q\}}$. On the other hand, we
will have $X_q \times V_{\min\{j'+k,q\}} \subseteq B_{j''-1}$ as soon
as $\min\{j'+k,q\} < j''$.  Since $j'' \leq q$, it follows from~\eqref{eq:j+s2} that $j'+k<q$. Hence, using~\eqref{eq:j+s2} once more, we see that
$\min\{j'+k,q\}=j'+k<j''$, as required.
\end{proof}

\begin{claim}\label{claim:AD}
For all $r\in \left[3^{-(M-1)/n},3^{q/n}\right]$ and all $x\in X$, we have
$$
\frac13 r^n\le\frac{\mu(B(x,r))}{\mu(X_q)}\le 4r^n.
$$
\end{claim}

\begin{proof}
By translation invariance we may assume that $x=0$. Let $j$ be the integer such that $3^{j/n}\le r<3^{(j+1)/n}$. Then $B(0,r)=B_j$. If $j\le 0$ then $B_j=W_j\times \{0\}$, and hence
\begin{equation}\label{eq:muBj, j neg}
\frac{\mu(B(x,r))}{\mu(X_q)}=\frac{\mu(B_j)}{\mu(X_q)}=3^j\in \left[\frac13 r^n,r^n\right].
\end{equation}
If $j\ge 1$ the it follows from~\eqref{eq:def Bj} that $B_j\supseteq X_q\times V_j$, and hence
\begin{equation}\label{muB_j, j pos lower bd}
\frac{\mu(B(0,r))}{\mu(X_q)}=\frac{\mu(B_j)}{\mu(X_q)}\ge \mu(V_j)=3^j\ge \frac13r^n.
\end{equation}
At the same time, it follows from~\eqref{eq:def Bj} that
\begin{multline}\label{eq:compute measure1} \frac{\mu(B(0,r))}{\mu(X_q)}= \frac{\mu(B_j)}{\mu(X_q)} \le 3^j +
\sum_{\ell=1}^{\min\{j+k,q\}} \frac{\mu(E_\ell)}{\mu(X_q)}\mu(V_\ell)\stackrel{\eqref{mujq}}{\leq}  3^j +\frac{2}{q}
\sum_{\ell=1}^{\min\{j+k,q\}}  3^j\\\le 3^j+\frac{3}{q}\cdot 3^{\min\{j+k,q\}}=3^j+\frac{3}{3^k}\cdot 3^{\min\{j+k,q\}}=4\cdot 3^j\le 4r^n,\end{multline}
as required.
\end{proof}
Claim~\ref{claim:AD} implies the second assertion of Lemma~\ref{lem:finite AD}, since $\mu(X_q)=3^M=3^{-a+1}$.

 We shall now prove the third assertion of Lemma~\ref{lem:finite AD}. Since the balls $B(x,r)$ in $X$ take the form $x+B_j$ for some $j$, we have
$$ M f(x) = \max_{-M \leq j \leq q} \frac{1}{\mu(B_j)} \sum_{y \in B_j} |f(x+y)|.$$
>From \eqref{mee} we can find $f_q: X_q \to \R_+$ with $\|f_q\|_{L_1(X_q)} = 1$, and $\lambda > 0$, such that
\begin{equation}\label{eq:assumption lower fq}
\mu\left( M_q f_q > \lambda \right) > \frac{q}{2\lambda}.
\end{equation}
We extend this function $f_q$ to a function $f: X \to \R_+$ defined
by $f(x,y) \stackrel{\mathrm{def}}{=} f_q(x)$ for $x \in X_q$ and $y \in \FF_3^{q}$.
Thus,
\begin{equation}\label{eq:f norm} \|f\|_{L_1(X)} = 3^{q}.
\end{equation}
Observe  that  for $1\le j\le q-k$ we have,
$$
\mu(B_j)\stackrel{\eqref{eq:compute measure1}}{\le} 4\mu(X_q)3^j\stackrel{\eqref{mujq}}{\le} 8q\mu(E_{j+k})\mu(V_j)=8\mu(E_{j+k})\mu(V_{j+k}).
$$
Hence, for all $(x,x')\in X$ we have,
$$ M f(x,x') \ge \max_{1 \leq j \leq q-k} \frac{1}{8\mu(V_{j+k})\mu(E_{j+k})}
\sum_{(y,y') \in B_j} |f_q(x+y)|.$$
Since $B_j$ contains $E_{j+k} \times V_{j+k}$ for $1\le j\le q-k$, we conclude that
\begin{multline}\label{substract the log} M f(x,x') \ge \max_{k+1 \leq \ell \leq q} \frac{1}{8\mu(E_\ell)}
\sum_{y \in E_\ell} |f_q(x+y)|\\\ge \frac{1}{8}\left(\max_{1 \leq j \leq q} \frac{1}{\mu(E_j)}
\sum_{y \in E_j} |f_q(x+y)|-\sum_{j=1}^{k}\frac{1}{\mu(E_j)}
\sum_{y \in E_j} |f_q(x+y)|\right).
\end{multline}
Denote $g:X\to \R$ by $g(x,x')=\sum_{j=1}^k\frac{1}{\mu(E_j)}
\sum_{y \in E_j} |f_q(x+y)|$. Then
\begin{equation}\label{eq:g norm}
\|g\|_{L_1(X)}\le k3^q\|f_q\|_{L_1(X_q)}\stackrel{\eqref{eq:f norm}}{=} \|f\|_{L_1(X)}\log_3 q.
\end{equation}
It follows from~\eqref{substract the log} that we have the pointwise bound $M_qf_q(x)\le 8 Mf(x,x')+g(x,x')$. Thus,
\begin{eqnarray*}
\frac{q\|f\|_{L_1(X)}}{2\lambda}&\stackrel{\eqref{eq:assumption lower fq}\wedge\eqref{eq:f norm}}{\le}&\mu\left((x,x')\in X:\ M_qf_q(x)>\lambda\right)\\&\le&\mu\left(8Mf+g>\lambda\right)\\&\le& \mu\left(Mf>\frac{\lambda}{16}\right)+\mu\left(g>\frac{\lambda}{2}\right)\\&\le& \mu\left(Mf>\frac{\lambda}{16}\right)+\frac{2\|g\|_{L_1(X)}}{\lambda}\\&\stackrel{\eqref{eq:g norm}}{\le}& \mu\left(Mf>\frac{\lambda}{16}\right)+\frac{2\log_3 q\|f\|_{L_1(X)}}{\lambda}.
\end{eqnarray*}
Hence,
\begin{equation}\label{eq:nlog n}
\| M \|_{L_1(X) \to L_{1,\infty}(X)} \gtrsim q\gtrsim n\log n,
\end{equation}
which gives the third assertion of Lemma~\ref{lem:finite AD}.

A similar argument (requiring a closer inspection of the details of Proposition \ref{prelim}) can be used to give the fourth assertion of Lemma~\ref{lem:finite AD}; alternatively, one can use \eqref{eq:nlog n} and the pigeonhole principle to show that a dilated version $M_{r \cdot 2^\Z}$ of the lacunary maximal function has weak $(1,1)$ norm $\gtrsim \log n$
for some $r > 0$, and then rescale the metric.  We omit the details.

It remains to verify the $L_p$ bound in assertion $(5)$ of Lemma~\ref{lem:finite AD}, i.e., to show for all $f \in L_p(X)$ we have
$$ \left\| \max_{-M \leq j \leq q} \frac{1}{\mu(B_j)} \sum_{y \in B_j} |f(x+y)| \right\|_{L_p(X)}
\lesssim_p \|f\|_{L_p(X)}.$$
The contribution of the case $-M \leq j \leq 0$ can be handled by Doob's maximal inequality as in Section~\ref{countersec}, so we need only consider the case $1 \leq j \leq q$.  Using~\eqref{eq:compute measure2} and the definition of $B_j$, we soon verify the pointwise estimate
\begin{multline}\label{eq:pointwise micro}
\max_{1 \leq j \leq q} \frac{1}{\mu(B_j)} \sum_{y \in B_j} |f(x+y)|
\lesssim \max_{1 \leq i \leq q} \frac{1}{\mu(X_q \times V_i)} \sum_{y \in X_q \times V_i} |f(x+y)| \\
 + \max_{1 \leq i \leq q} \frac{1}{\mu(E_i \times V_i)} \sum_{y \in E_i \times V_i} |f(x+y)|.
\end{multline}
Indeed, denote
$$
h(x)=\max_{1 \leq i \leq q} \frac{1}{\mu(X_q \times V_i)} \sum_{y \in X_q \times V_i} |f(x+y)|
 + \max_{1 \leq i \leq q} \frac{1}{\mu(E_i \times V_i)} \sum_{y \in E_i \times V_i} |f(x+y)|.
$$
Then for all $1\le j\le q$,
\begin{multline*}
\frac{1}{\mu(B_j)} \sum_{y \in B_j} |f(x+y)|\stackrel{\eqref{eq:def Bj}}{\lesssim} \frac{\sum_{y\in X_q\times V_j}|f(x+y)|+\sum_{\ell=1}^{\min\{j+k,q\}}\sum_{y\in E_\ell\times V_\ell}|f(x+y)|}{\mu(B_j)}\\
\stackrel{\eqref{muB_j, j pos lower bd}}{\le} \frac{h(x)\mu\left(X_q\times V_j\right)+\sum_{\ell=1}^{\min\{j+k,q\}}h(x)\mu\left(E_\ell\times V_\ell\right)}{3^j\mu(X_q)}\stackrel{\eqref{eq:compute measure1}}{\le} 4h(x),
\end{multline*}
proving~\eqref{eq:pointwise micro}.

The fact that the first term in the right-hand side of~\eqref{eq:pointwise micro} is bounded in $L_p(X)$ again follows from Doob's maximal
inequality, while the $L_p(X)$ boundedness of the second term in the right-hand side of~\eqref{eq:pointwise micro} follows from
\eqref{mefp}, Doob's maximal inequality and a Fubini
argument, as in Section~\ref{countersec}. The proof of Theorem
\ref{micro-ex} is now complete. \end{proof}



\bibliographystyle{abbrv}
\bibliography{HLM}

\end{document}